\documentclass{amsart}
\usepackage{graphicx}

\usepackage{amsmath}
\usepackage{amsfonts}
\usepackage{amssymb}
\usepackage{wasysym}
\usepackage{enumerate}
\usepackage{longtable}
\usepackage{pgf,tikz}
\usepackage{cite, url}
\usepackage{extarrows, cancel}

\usetikzlibrary{arrows}
\usetikzlibrary{arrows.meta} 

\theoremstyle{plain}
\newtheorem{theor}{Theorem}

\newtheorem{nlemma}[theor]{Lemma}
\newtheorem{kor}[theor]{Corollary}
\newtheorem{nconjecture}[theor]{Conjecture}
\theoremstyle{definition}
\newtheorem{definition}[theor]{Definition}

\newtheorem{remark}[theor]{Remark}

\newcommand{\dcup}{\mathrel{\dot{\cup}}}
\newcommand{\nsqb}{\not\sqsubseteq}

\newcommand\blfootnote[1]{%
  \begingroup
  \renewcommand\thefootnote{}\footnote{#1}%
  \addtocounter{footnote}{-1}%
  \endgroup
}

\begin{document}

\title{
Definability in the substructure ordering of finite directed graphs 
}


\author{\'{A}d\'{a}m Kunos}


\maketitle

\begin{abstract}
We deal with first-order definability in the {\it substructure} ordering $(\mathcal{D}; \sqsubseteq)$ of finite directed graphs. 
In two papers, the author has already investigated the first-order language of the {\it embeddability} ordering $( \mathcal{D}; \leq)$. 
The latter has turned out to be quite strong, e.g., it has been shown that, modulo edge-reversing (on the whole graphs), it can express the full second-order language of directed graphs. 
Now we show that, with finitely many directed graphs added as constants, the first order language of  $( \mathcal{D}; \sqsubseteq)$ can express that of $( \mathcal{D}; \leq)$. 

The limits of the expressive power of such languages are intimately related to the automorphism groups of the orderings. 
Previously, analogue investigations have found the concerning automorphism groups to be quite trivial, e.g., the automorphism group of $( \mathcal{D}; \leq)$ is isomorphic to $\mathbb{Z}_2$. 
Here, unprecedentedly, this is not the case. 
Even though we conjecture that the automorphism group is isomorphic to $(\mathbb{Z}_2^4 \times S_4)\rtimes_{\alpha} \mathbb{Z}_2$, with a particular $\alpha$ in the semidirect product, we only prove it is finite.
\end{abstract}

\section{Introduction and formulation of our main theorems}

\blfootnote{This research was supported by the UNKP-17-3 New National Excellence Program of the Ministry of Human Capacities and by the Hungarian National Foundation for Scientific Research grant no. K115518.}

In 2009--2010 J. Je\v{z}ek and R. McKenzie published a series of papers 
\cite{Jezek2009_1, Jezek2010, Jezek2009_3, Jezek2009_4} 
in which they have examined (among other things) the first-order definability in the substructure orderings of finite mathematical structures with a given type, and determined the automorphism group of these orderings. 
They considered finite semilattices \cite{Jezek2009_1}, ordered sets \cite{Jezek2010}, distributive lattices \cite{Jezek2009_3} and lattices \cite{Jezek2009_4}. 
Similar investigations \cite{Kunos2015, KunosII, Wires2016, Ramanujam2016, Thinniyam2017, lmcs:4846} have emerged since. 
The current paper is one of such, connected strongly to the author's papers \cite{Kunos2015, KunosII} that dealt with the {\it embeddability} ordering of finite directed graphs. Now, instead of embeddability, we are examining the {\it substructure} ordering of finite directed graphs.

Let us consider a nonempty set $V$ and a binary relation $E\subseteq V^2$. We call the pair $G=(V,E)$ a {\it directed graph} or just {\it digraph}.
Let $\mathcal{D}$\label{xcvshgw8} denote the set of isomorphism types of finite digraphs. 
The elements of $V(=V(G))$\label{qwqepwrpdg} and $E(=E(G))$\label{qweoweru3} are called the{ \it vertices} and {\it edges} of $G$, respectively.
A digraph $G$ is said to be {\it embeddable} into $G'$, and we write $G\leq G'$, if there exists an injective homomorphism $\varphi : G \to G'$, i.e. an injective map for which $(v_1, v_2)\in E(G)$ implies $(\varphi(v_1), \varphi(v_2))\in E(G')$.
A digraph $G$ is a {\it substructure} of $G'$, and we write $G\sqsubseteq G'$, if it is isomorphic to an induced substructure (on some subset of the vertices) of $G'$ .
Note that in graph theory the term {\it subgraph} is used rather for the embeddability concept, thus we opt not to use it at all in this paper. 
For the most part, we will use the substructure concept, and will just use the term {\it substructure}, correspondingly.
Every substructure is embeddable but the converse is not true.
The names of these two concepts often mix both orally and on paper when it is clear from the context which notion we are using the whole time. 
In the present paper, however, we must be very cautious as both concepts are used alternately throughout the whole paper.
It is easy to see that both $\leq$ and $\sqsubseteq$ are partial orders on $\mathcal{D}$. 
Both partially ordered sets are naturally graded. The digraph $G$ is on the $n$th level of $(\mathcal{D}; \leq)$ or $(\mathcal{D}; \sqsubseteq)$ if $|V(G)|+|E(G)|=n$ or $|V(G)|=n$, respectively.
See Figures \ref{nagyabra} and \ref{41511} for the initial segments of the Hasse diagrams of the two partial orders.

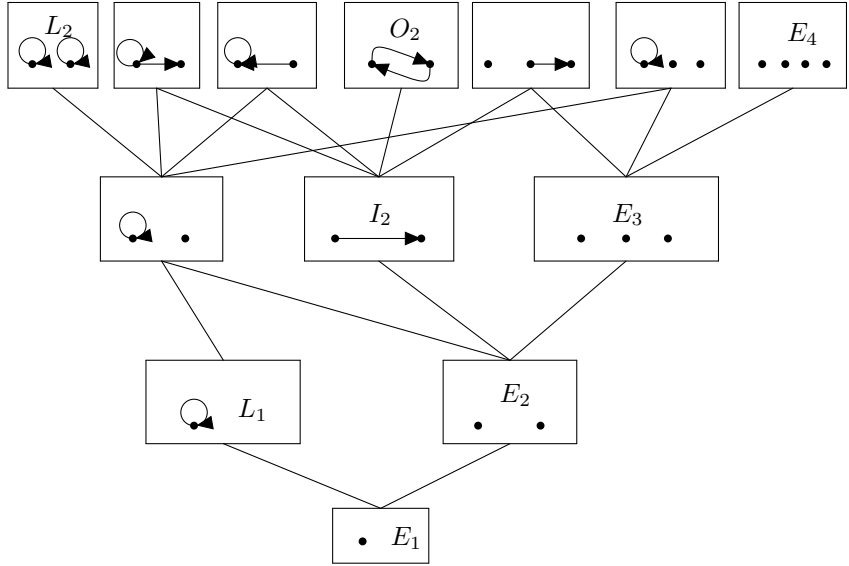
\begin{figure}[h]
\begin{center}
\begin{tikzpicture}[line cap=round,line join=round,>=triangle 45,x=0.83cm,y=1.0cm]
\clip(0.9,0.91) rectangle (15.1,9);
\draw (6.47,2.22)-- (8.01,2.22);
\draw (8.01,2.22)-- (8.01,1.49);
\draw (8.01,1.49)-- (6.47,1.49);
\draw (6.47,1.49)-- (6.47,2.22);
\draw [->] (6.51,5.81) -- (7.89,5.81);
\draw [->] (3.33,8.13) -- (4.04,8.13);
\draw [->] (5.84,8.13) -- (4.95,8.13);
\draw [->] (7.1,8.13) .. controls +(up:0.3cm) .. (8.03,8.13);
\draw [->] (8.03,8.13) .. controls +(down:0.3cm) .. (7.1,8.13);
\draw [->] (9.64,8.13) -- (10.29,8.13);
\draw (2.97,7.81)-- (4.34,7.81);
\draw (4.34,7.81)-- (4.34,8.95);
\draw (4.34,8.95)-- (2.97,8.95);
\draw (2.97,8.95)-- (2.97,7.81);
\draw (4.63,7.81)-- (6.21,7.81);
\draw (6.21,7.81)-- (6.21,8.95);
\draw (6.21,8.95)-- (4.63,8.95);
\draw (4.63,8.95)-- (4.63,7.81);
\draw (6.66,7.81)-- (8.48,7.81);
\draw (8.48,7.81)-- (8.48,8.95);
\draw (8.48,8.95)-- (6.66,8.95);
\draw (6.66,8.95)-- (6.66,7.81);
\draw (8.71,7.81)-- (10.58,7.81);
\draw (10.58,7.81)-- (10.58,8.95);
\draw (10.58,8.95)-- (8.71,8.95);
\draw (8.71,8.95)-- (8.71,7.81);
\draw (11.01,7.81)-- (12.76,7.81);
\draw (12.76,7.81)-- (12.76,8.95);
\draw (12.76,8.95)-- (11.01,8.95);
\draw (11.01,8.95)-- (11.01,7.81);
\draw (12.98,7.81)-- (14.7,7.81);
\draw (14.7,7.81)-- (14.7,8.95);
\draw (14.7,8.95)-- (12.98,8.95);
\draw (12.98,8.95)-- (12.98,7.81);
\draw (2.75,5.51)-- (4.71,5.51);
\draw (4.71,5.51)-- (4.71,6.63);
\draw (4.71,6.63)-- (2.75,6.63);
\draw (2.75,6.63)-- (2.75,5.51);
\draw (6.02,5.51)-- (8.41,5.51);
\draw (8.41,5.51)-- (8.41,6.63);
\draw (8.41,6.63)-- (6.02,6.63);
\draw (6.02,6.63)-- (6.02,5.51);
\draw (9.7,5.51)-- (12.65,5.51);
\draw (12.65,5.51)-- (12.65,6.63);
\draw (12.65,6.63)-- (9.7,6.63);
\draw (9.7,6.63)-- (9.7,5.51);
\draw (3.48,3.08)-- (5.95,3.08);
\draw (5.95,3.08)-- (5.95,4.19);
\draw (5.95,4.19)-- (3.48,4.19);
\draw (3.48,4.19)-- (3.48,3.08);
\draw (8.24,3.08)-- (10.38,3.08);
\draw (10.38,3.08)-- (10.38,4.19);
\draw (10.38,4.19)-- (8.24,4.19);
\draw (8.24,4.19)-- (8.24,3.08);
\draw (7.24,2.22)-- (4.72,3.08);
\draw (7.24,2.22)-- (9.31,3.08);
\draw (9.31,4.19)-- (11.17,5.51);
\draw (4.72,4.19)-- (3.73,5.51);
\draw (9.31,4.19)-- (7.21,5.51);
\draw (9.31,4.19)-- (3.73,5.51);
\draw (3.73,6.63)-- (3.65,7.81);
\draw (3.73,6.63)-- (5.42,7.81);
\draw (7.21,6.63)-- (3.65,7.81);
\draw (7.21,6.63)-- (5.42,7.81);
\draw (7.21,6.63)-- (7.57,7.81);
\draw (7.21,6.63)-- (9.65,7.81);
\draw (11.17,6.63)-- (13.84,7.81);
\draw (11.17,6.63)-- (11.89,7.81);
\draw (11.17,6.63)-- (9.65,7.81);
\draw (7.25,2.1) node[anchor=north west] {$E_1$};
\draw (7.23,8.86) node[anchor=north west] {$O_2$};
\draw (13.6,8.8) node[anchor=north west] {$E_4$};
\draw (6.9,6.4) node[anchor=north west] {$I_2$};
\draw (10.8,6.4) node[anchor=north west] {$E_3$};
\draw (1.27,7.81)-- (2.71,7.81);
\draw (2.71,7.81)-- (2.71,8.95);
\draw (2.71,8.95)-- (1.27,8.95);
\draw (1.27,8.95)-- (1.27,7.81);
\draw (1.7,8.9) node[anchor=north west] {$L_2$};
\draw (3.73,6.63)-- (1.99,7.81);
\draw (3.73,6.63)-- (11.89,7.81);
\draw (9,4) node[anchor=north west] {$E_2$};
\draw (4.8,3.8) node[anchor=north west] {$L_1$};
\begin{scriptsize}
\fill [color=black] (-0.72,3.32) circle (1.5pt);
\draw[color=black] (-0.03,5.91) node {$F$};
\fill [color=black] (-0.98,5.81) circle (1.5pt);
\draw[color=black] (0.06,8.45) node {$G$};
\fill [color=black] (-1.24,8.13) circle (1.5pt);
\draw[color=black] (-0.01,10.81) node {$H$};
\fill [color=black] (4.25,11.07) circle (1.5pt);
\draw[color=black] (4.4,11.33) node {$E$};
\fill [color=black] (8.8,10.99) circle (1.5pt);
\draw[color=black] (8.9,11.25) node {$I$};
\fill [color=black] (9.8,11.05) circle (1.5pt);
\draw[color=black] (9.94,11.31) node {$J$};
\fill [color=black] (3.27,12.38) circle (1.5pt);
\draw[color=black] (3.42,12.64) node {$K$};
\fill [color=black] (4.11,12.32) circle (1.5pt);
\draw[color=black] (4.24,12.57) node {$L$};
\fill [color=black] (6.51,12.48) circle (1.5pt);
\draw[color=black] (6.67,12.74) node {$M$};
\fill [color=black] (7.89,12.43) circle (1.5pt);
\draw[color=black] (8.04,12.7) node {$N$};
\fill [color=black] (10.45,12.24) circle (1.5pt);
\draw[color=black] (10.6,12.5) node {$O$};
\fill [color=black] (11.17,12.22) circle (1.5pt);
\draw[color=black] (11.31,12.48) node {$P$};
\fill [color=black] (11.84,12.2) circle (1.5pt);
\draw[color=black] (12,12.46) node {$Q$};
\fill [color=black] (3.33,13.49) circle (1.5pt);
\draw[color=black] (3.48,13.74) node {$R$};
\fill [color=black] (4.04,13.51) circle (1.5pt);
\draw[color=black] (4.2,13.78) node {$S$};
\fill [color=black] (4.95,13.76) circle (1.5pt);
\draw[color=black] (5.07,14.02) node {$T$};
\fill [color=black] (5.84,13.68) circle (1.5pt);
\draw[color=black] (6,13.94) node {$U$};
\fill [color=black] (7.1,13.85) circle (1.5pt);
\draw[color=black] (7.23,14.11) node {$V$};
\fill [color=black] (8.03,13.95) circle (1.5pt);
\draw[color=black] (8.21,14.22) node {$W$};
\fill [color=black] (8.96,13.77) circle (1.5pt);
\draw[color=black] (9.1,14.04) node {$Z$};
\fill [color=black] (9.64,13.67) circle (1.5pt);
\draw[color=black] (9.9,13.94) node {$A_1$};
\fill [color=black] (10.29,13.67) circle (1.5pt);
\draw[color=black] (10.57,13.94) node {$B_1$};
\fill [color=black] (11.46,13.78) circle (1.5pt);
\draw[color=black] (11.74,14.04) node {$C_1$};
\fill [color=black] (11.92,13.8) circle (1.5pt);
\draw[color=black] (12.2,14.05) node {$D_1$};
\fill [color=black] (12.36,13.78) circle (1.5pt);
\draw[color=black] (12.65,14.04) node {$E_1$};
\fill [color=black] (13.35,13.95) circle (1.5pt);
\draw[color=black] (13.61,14.2) node {$F_1$};
\fill [color=black] (13.72,13.95) circle (1.5pt);
\draw[color=black] (14,14.2) node {$G_1$};
\fill [color=black] (14.04,13.97) circle (1.5pt);
\draw[color=black] (14.3,14.2) node {$L_1$};
\fill [color=black] (14.38,14) circle (1.5pt);
\draw[color=black] (14.65,14.2) node {$M_1$};
\fill [color=black] (4.25,3.32) circle (1.5pt);
\draw[color=black] (4.53,3.59) ;
\fill [color=black] (8.8,3.32) circle (1.5pt);
\fill [color=black] (9.8,3.32) circle (1.5pt);
\fill [color=black] (3.27,5.81) circle (1.5pt);
\draw[color=black] (3.55,6.07) ;
\fill [color=black] (6.51,5.81) circle (1.5pt);
\fill [color=black] (7.89,5.81) circle (1.5pt);
\fill [color=black] (10.45,5.81) circle (1.5pt);
\fill [color=black] (11.17,5.81) circle (1.5pt);
\fill [color=black] (11.84,5.81) circle (1.5pt);
\fill [color=black] (4.11,5.81) circle (1.5pt);
\fill [color=black] (3.33,8.13) circle (1.5pt);
\draw[color=black] (3.59,8.39) ;
\fill [color=black] (4.04,8.13) circle (1.5pt);
\fill [color=black] (4.95,8.13) circle (1.5pt);
\draw[color=black] (5.22,8.39) ;
\fill [color=black] (5.84,8.13) circle (1.5pt);
\fill [color=black] (7.1,8.13) circle (1.5pt);
\fill [color=black] (8.03,8.13) circle (1.5pt);
\fill [color=black] (8.96,8.13) circle (1.5pt);
\fill [color=black] (9.64,8.13) circle (1.5pt);
\fill [color=black] (10.29,8.13) circle (1.5pt);
\fill [color=black] (11.46,8.13) circle (1.5pt);
\draw[color=black] (11.68,8.39) ;
\fill [color=black] (11.92,8.13) circle (1.5pt);
\fill [color=black] (12.36,8.13) circle (1.5pt);
\fill [color=black] (13.35,8.13) circle (1.5pt);
\fill [color=black] (13.72,8.13) circle (1.5pt);
\fill [color=black] (14.04,8.13) circle (1.5pt);
\fill [color=black] (14.38,8.13) circle (1.5pt);
\fill [color=black] (6.95,1.78) circle (1.5pt);
\fill [color=black] (-2.48,8.95) circle (1.5pt);
\draw[color=black] (0.1,7.78) node {$T_2$};
\fill [color=black] (-2.32,7.81) circle (1.5pt);
\draw[color=black] (0.21,7.09) node {$U_2$};
\fill [color=black] (-2.7,6.63) circle (1.5pt);
\draw[color=black] (0.1,6.37) node {$V_2$};
\fill [color=black] (-2.66,5.51) circle (1.5pt);
\draw[color=black] (0.17,5.7) node {$W_2$};
\fill [color=black] (-2.38,4.19) circle (1.5pt);
\draw[color=black] (0.15,4.92) node {$Z_2$};
\fill [color=black] (-2.34,3.08) circle (1.5pt);
\draw[color=black] (0.17,4.25) node {$A_3$};
\fill [color=black] (2.97,-0.76) circle (1.5pt);
\draw[color=black] (3.25,-0.5) node {$B_3$};
\fill [color=black] (4.34,-0.72) circle (1.5pt);
\draw[color=black] (4.61,-0.46) node {$C_3$};
\fill [color=black] (4.63,-0.69) circle (1.5pt);
\draw[color=black] (4.91,-0.42) node {$D_3$};
\fill [color=black] (6.21,-0.82) circle (1.5pt);
\draw[color=black] (6.43,-0.5) node {$E_3$};
\fill [color=black] (6.66,-0.79) circle (1.5pt);
\draw[color=black] (6.93,-0.54) node {$F_3$};
\fill [color=black] (8.48,-0.82) circle (1.5pt);
\draw[color=black] (8.77,-0.55) node {$G_3$};
\fill [color=black] (8.71,-0.94) circle (1.5pt);
\draw[color=black] (8.99,-0.68) node {$H_3$};
\fill [color=black] (10.58,-0.86) circle (1.5pt);
\draw[color=black] (10.81,-0.59) node {$I_3$};
\fill [color=black] (11.01,-0.85) circle (1.5pt);
\draw[color=black] (11.27,-0.59) node {$J_3$};
\fill [color=black] (12.76,-0.81) circle (1.5pt);
\draw[color=black] (13.04,-0.55) node {$K_3$};
\fill [color=black] (12.98,-0.77) circle (1.5pt);
\draw[color=black] (13.24,-0.5) node {$L_3$};
\fill [color=black] (14.7,-0.73) circle (1.5pt);
\draw[color=black] (14.97,-0.46) node {$M_3$};
\fill [color=black] (2.75,-2.01) circle (1.5pt);
\draw[color=black] (3.03,-1.76) node {$N_4$};
\fill [color=black] (4.71,-2.05) circle (1.5pt);
\draw[color=black] (5,-1.8) node {$O_4$};
\fill [color=black] (6.02,-2.07) circle (1.5pt);
\draw[color=black] (6.3,-1.82) node {$P_4$};
\fill [color=black] (8.41,-2.07) circle (1.5pt);
\draw[color=black] (8.69,-1.82) node {$Q_4$};
\fill [color=black] (9.7,-1.98) circle (1.5pt);
\draw[color=black] (9.97,-1.72) node {$R_4$};
\fill [color=black] (12.65,-2) circle (1.5pt);
\draw[color=black] (12.92,-1.74) node {$S_4$};
\fill [color=black] (3.48,-3.15) circle (1.5pt);
\draw[color=black] (3.83,-1.84) node {$H_5$};
\fill [color=black] (5.95,-3.32) circle (1.5pt);
\draw[color=black] (6.28,-1.82) node {$I_5$};
\fill [color=black] (8.24,-3.38) circle (1.5pt);
\draw[color=black] (8.62,-1.8) node {$J_5$};
\fill [color=black] (10.38,-3.42) circle (1.5pt);
\draw[color=black] (10.77,-1.84) node {$K_5$};
\fill [color=black] (2.27,13.81) circle (1.5pt);
\draw[color=black] (2.53,14.07) node {$J_2$};
\fill [color=black] (1.66,13.83) circle (1.5pt);
\draw[color=black] (1.93,14.09) node {$S_2$};
\fill [color=black] (1.66,8.13) circle (1.5pt);
\draw[color=black] (1.92,8.39) ;
\fill [color=black] (2.27,8.13) circle (1.5pt);
\draw[color=black] (2.55,8.39) ;
\fill [color=black] (2.71,-0.61) circle (1.5pt);
\draw[color=black] (2.99,-0.35) node {$N_6$};
\fill [color=black] (1.27,-0.78) circle (1.5pt);
\draw[color=black] (1.55,-0.52) node {$O_6$};
\draw[rotate around={-90:(4.25,3.32)}] [->] (4.25,3.32) arc (360:10:5pt);
\draw[rotate around={-90:(3.27,5.81)}] [->] (3.27,5.81) arc (360:10:5pt);
\draw[rotate around={-90:(1.66,8.13)}] [->] (1.66,8.13) arc (360:10:5pt);
\draw[rotate around={-90:(2.27,8.13)}] [->] (2.27,8.13) arc (360:10:5pt);
\draw[rotate around={-60:(3.33,8.13)}] [->] (3.33,8.13) arc (360:10:5pt);
\draw[rotate around={-90:(11.46,8.13)}] [->] (11.46,8.13) arc (360:10:5pt);
\draw[rotate around={-90:(4.95,8.13)}] [->] (4.95,8.13) arc (360:10:5pt);
\end{scriptsize}
\end{tikzpicture}
\end{center}
\caption{The initial segment of the Hasse diagram of $(\mathcal{D}; \leq)$.}
\label{nagyabra}
\end{figure}

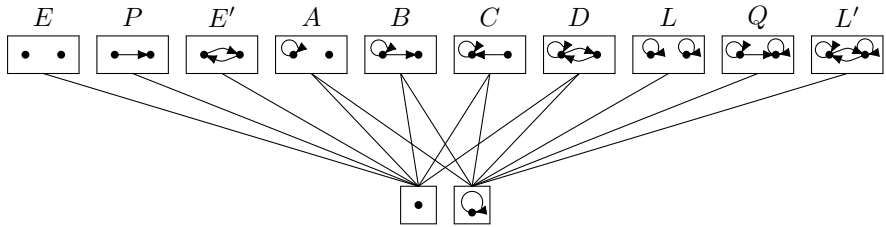
\begin{figure}[h]
\begin{center}
\begin{tikzpicture}[line cap=round,line join=round,>=Latex, x=.475cm,y=.5cm]
\clip(-0.01,.95) rectangle (30,6.8);

\draw (0,5)-- (2,5);
\draw (2,5)-- (2,6);
\draw (2,6)-- (0,6);
\draw (0,6)-- (0,5);
\fill [color=black] (0.5,5.5) circle (1.5pt);
\fill [color=black] (1.5,5.5) circle (1.5pt);
\draw (1,5)-- (11.5, 2);
\draw[color=black] (1,6.5) node {$E$};

\draw (2.5,5)-- (4.5,5);
\draw (4.5,5)-- (4.5,6);
\draw (4.5,6)-- (2.5,6);
\draw (2.5,6)-- (2.5,5);
\fill [color=black] (3,5.5) circle (1.5pt);
\fill [color=black] (4,5.5) circle (1.5pt);
\draw [->] (3,5.5) -- (4,5.5);
\draw (3.5,5)-- (11.5, 2);
\draw[color=black] (3.5,6.5) node {$P$};

\draw (5,5)-- (7,5);
\draw (7,5)-- (7,6);
\draw (7,6)-- (5,6);
\draw (5,6)-- (5,5);
\fill [color=black] (5.5,5.5) circle (1.5pt);
\fill [color=black] (6.5,5.5) circle (1.5pt);
\draw [->] (5.5,5.5) ..controls(6, 5.7)..  (6.5,5.5);
\draw [->] (6.5,5.5) ..controls(6, 5.3)..  (5.5,5.5);
\draw (6,5)-- (11.5, 2);
\draw[color=black] (6,6.5) node {$E'$};

\draw (7.5,5)-- (9.5,5);
\draw (9.5,5)-- (9.5,6);
\draw (9.5,6)-- (7.5,6);
\draw (7.5,6)-- (7.5,5);
\fill [color=black] (8,5.5) circle (1.5pt);
\draw[rotate around={-60:(8,5.5)}] [->] (8,5.5) arc (360:10:3pt);
\fill [color=black] (9,5.5) circle (1.5pt);
\draw (8.5,5)-- (11.5, 2);
\draw (8.5,5)-- (13, 2);
\draw[color=black] (8.5,6.5) node {$A$};

\draw (10,5)-- (12,5);
\draw (12,5)-- (12,6);
\draw (12,6)-- (10,6);
\draw (10,6)-- (10,5);
\fill [color=black] (10.5,5.5) circle (1.5pt);
\draw[rotate around={-60:(10.5,5.5)}] [->] (10.5,5.5) arc (360:10:3pt);
\fill [color=black] (11.5,5.5) circle (1.5pt);
\draw [->] (10.5,5.5) -- (11.5,5.5);
\draw (11,5)-- (11.5, 2);
\draw (11,5)-- (13, 2);
\draw[color=black] (11,6.5) node {$B$};

\draw (12.5,5)-- (14.5,5);
\draw (14.5,5)-- (14.5,6);
\draw (14.5,6)-- (12.5,6);
\draw (12.5,6)-- (12.5,5);
\fill [color=black] (13,5.5) circle (1.5pt);
\draw[rotate around={-45:(13,5.5)}] [->] (13,5.5) arc (360:10:3pt);
\fill [color=black] (14,5.5) circle (1.5pt);
\draw [->] (14,5.5) -- (13,5.5);
\draw (13.5,5)-- (11.5, 2);
\draw (13.5,5)-- (13, 2);
\draw[color=black] (13.5,6.5) node {$C$};

\draw (15,5)-- (17,5);
\draw (17,5)-- (17,6);
\draw (17,6)-- (15,6);
\draw (15,6)-- (15,5);
\fill [color=black] (15.5,5.5) circle (1.5pt);
\draw[rotate around={-40:(15.5,5.5)}] [->] (15.5,5.5) arc (360:10:3pt);
\fill [color=black] (16.5,5.5) circle (1.5pt);
\draw [->] (15.5,5.5) ..controls(16, 5.7)..  (16.5,5.5);
\draw [->] (16.5,5.5) ..controls(16, 5.3)..  (15.5,5.5);
\draw (16,5)-- (11.5, 2);
\draw (16,5)-- (13, 2);
\draw[color=black] (16,6.5) node {$D$};

\draw (17.5,5)-- (19.5,5);
\draw (19.5,5)-- (19.5,6);
\draw (19.5,6)-- (17.5,6);
\draw (17.5,6)-- (17.5,5);
\fill [color=black] (18,5.5) circle (1.5pt);
\draw[rotate around={-90:(18,5.5)}] [->] (18,5.5) arc (360:10:3pt);
\fill [color=black] (19,5.5) circle (1.5pt);
\draw[rotate around={-90:(19,5.5)}] [->] (19,5.5) arc (360:10:3pt);
\draw (18.5,5)-- (13, 2);
\draw[color=black] (18.5,6.5) node {$L$};

\draw (20,5)-- (22,5);
\draw (22,5)-- (22,6);
\draw (22,6)-- (20,6);
\draw (20,6)-- (20,5);
\fill [color=black] (20.5,5.5) circle (1.5pt);
\draw[rotate around={-40:(20.5,5.5)}] [->] (20.5,5.5) arc (360:10:3pt);
\fill [color=black] (21.5,5.5) circle (1.5pt);
\draw[rotate around={-90:(21.5,5.5)}] [->] (21.5,5.5) arc (360:10:3pt);
\draw [->] (20.5,5.5) -- (21.5,5.5);
\draw (21,5)-- (13, 2);
\draw[color=black] (21,6.5) node {$Q$};

\draw (22.5,5)-- (24.5,5);
\draw (24.5,5)-- (24.5,6);
\draw (24.5,6)-- (22.5,6);
\draw (22.5,6)-- (22.5,5);
\fill [color=black] (23,5.5) circle (1.5pt);
\draw[rotate around={-40:(23,5.5)}] [->] (23,5.5) arc (360:10:3pt);
\fill [color=black] (24,5.5) circle (1.5pt);
\draw[rotate around={-90:(24,5.5)}] [->] (24,5.5) arc (360:10:3pt);
\draw [->] (23,5.5) ..controls(23.5, 5.7)..  (24,5.5);
\draw [->] (24,5.5) ..controls(23.5, 5.3)..  (23,5.5);
\draw (23.5,5)-- (13, 2);
\draw[color=black] (23.5,6.5) node {$L'$};

\draw (12,1)-- (12,2);
\draw (12,2)-- (11,2);
\draw (11,2)-- (11,1);
\draw (11,1)-- (12,1);
\fill [color=black] (11.5,1.5) circle (1.5pt);

\draw (13.5,1)-- (13.5,2);
\draw (13.5,2)-- (12.5,2);
\draw (12.5,2)-- (12.5,1);
\draw (12.5,1)-- (13.5,1);
\fill [color=black] (13,1.3) circle (1.5pt);
\draw[rotate around={-90:(13,1.3)}] [->] (13,1.3) arc (360:10:4pt);

\end{tikzpicture}
\end{center}
\caption{The initial segment of the Hasse diagram of $(\mathcal{D}; \sqsubseteq)$. (Please disregard the labels unless they are being specifically referred to.)}
\label{41511}
\end{figure}

Let $(\mathcal{A};\leq)$ be an arbitrary poset.
An $n$-ary relation $R$ is said to be (first-order) definable in $(\mathcal{A};\leq)$ if there exists a first-order formula
$\Psi(x_1,x_2,\dots, x_n)$ with free variables $x_1,x_2,\dots, x_n$ in the language of partially ordered sets such that for any  $a_1,a_2,\dots, a_n\in \mathcal{A}$,
$\Psi(a_1,a_2,\dots, a_n)$  holds in $(\mathcal{A};\leq)$ if and only if $(a_1,a_2,\dots, a_n)\in R$.
A subset of $\mathcal{A}$ is definable if it is definable as a unary relation. An element $a\in \mathcal{A}$ is said to be definable if
the set $\{a\}$ is definable.

Our main result is the following.

\begin{theor}\label{46031} There exists a finite set of finite directed graphs $\{C_1, \dots, C_k\}$ such that the binary embeddability relation,
$$\{(G,G'): G\leq G'\},$$
is definable in the first-order language of $(\mathcal{D}; \sqsubseteq, C_1, \dots, C_k)$. Consequently, every relation definable in the first-order language of $(\mathcal{D}; \leq)$ is definable in that of $(\mathcal{D}; \sqsubseteq, C_1, \dots, C_k)$.
\end{theor}

In itself, this theorem is quite weightless, what fills it with content is that we already know \cite{Kunos2015, KunosII} that the first-order language of $(\mathcal{D}; \leq)$ is surprisingly strong. The paper \cite{Kunos2015} has two parts. The first deals with definability in $(\mathcal{D}; \leq)$, the second determines the automorphism group of $(\mathcal{D}; \leq)$ (building on the first part, of course). 
The paper \cite{KunosII} extends the main result of the first part of \cite{Kunos2015}, hence if one is only interested in definability, it is enough to read \cite{KunosII}. 
The main result there \cite[Theorem 5]{KunosII} is some kind of a characterization of the first-order definable relations in $(\mathcal{D}; \leq)$. 
To even state the result precisely, there is a 3-page-long preparation which we don't repeat here. 
We only provide some corollaries, demonstrating the power of definability in $(\mathcal{D}; \leq)$. 
With Theorem \ref{46031}, these corollaries transform immediately into statements for the first-order language of $(\mathcal{D}; \sqsubseteq, C_1, \dots, C_k)$. 
As this paper is about the substructure ordering, we formulate these versions, rather than the versions talking about $(\mathcal{D}; \leq)$.

\begin{kor}\label{58562} There exists a finite set of finite directed graphs $\{C_1, \dots, C_k\}$ such that  in the first-order language of $(\mathcal{D}; \sqsubseteq, C_1, \dots, C_k)$
\begin{itemize}
\item every single digraph $G$ is definable,
\item the set of weakly connected digraphs is definable, moreover,
\item the full second-order language of digraphs becomes available.
\end{itemize}
\end{kor}

Again, for the full scope of Theorem \ref{46031}, see \cite[Section 2]{KunosII}.

We remark that the notations of Theorem \ref{46031} and Corollary \ref{58562} may suggest that the set $\{C_1, \dots, C_k\}$ in the two statements can be the same.
This is not necessarily true, even though there is a strong connection between the two sets.
Depending on the set of Theorem \ref{46031}, an additional digraph might have to be added the get the corresponding set of Corollary \ref{58562}.
This is due to the fact that the first-order language of $(\mathcal{D}; \leq)$ does not yield the listed statements of Corollary \ref{58562} in itself. 
A constant (a particular digraph) needs to be added to the first-order language of $(\mathcal{D}; \leq)$ to make these true.
If this constant is not already there in the set of $\{C_1, \dots, C_k\}$ of Theorem \ref{46031} then its addition might be required to get that of Corollary \ref{58562}. 
As the equality of the sets is not stated anywhere, this is not a problem.
This affair is actually about $(\mathcal{D}; \leq)$, which is not the subject of our investigation here.
The interested reader should consult the first two sections of \cite{KunosII}.

We wish to make another remark on the lists $\{C_1, \dots, C_k\}$ to avoid false expectations. 
Naturally, as we proceed with our proof the lists $\{C_1, \dots, C_k\}$ will be continuously growing.
The final list is revealed late in the paper, and that is why we now outline it in advance.
To do so, we describe a family of our arguments used in the last, technical section of the paper.
Some properties of digraphs can be told by saying something about the list of their, say, at most 4-element substructures (without multiplicity, naturally).
For example one can tell if a digraph has loops based on the list of its 1-element substructures. 
Similarly, one can judge if it has a non-loop edge by the list of its (at most) 2-element substructures.
Far more complicated properties can be told in this way, say, {\it locally}.
We adopt this thinking in the last section of the paper.
This will force our lists  $\{C_1, \dots, C_k\}$ to be $\{ \text{at most 4-element digraphs}\}$.
This list is long but finite nevertheless.

The papers \cite{Jezek2009_1, Jezek2010, Jezek2009_3, Jezek2009_4, Kunos2015, Wires2016}, beyond dealing with definability, determined the automorphism groups of the orderings in question. 
In every case, the automorphisms came naturally and the automorphism groups were either trivial or isomorphic to $\mathbb{Z}_2$.
Despite all expectations, the partially ordered set $(\mathcal{D}, \sqsubseteq)$ stands out in that aspect. 
There are automorphisms far from trivial. 
Unfortunately, we are not able to determine the automorphism group, we can only prove it is finite.

\begin{theor}\label{15886}  The automorphism group of $(\mathcal{D}, \sqsubseteq)$ is finite.
\end{theor}

Even though we can not prove it, we formulate a conjecture for the automorphism group.

In Section 2, we prove Theorem \ref{15886}, and tell our conjecture on the automorphism group in detail. Section 3 contains the proof of Theorem \ref{46031} without some technicalities. In Section 4, the reader finds the technicalities skipped in Section 3.

\section{On the automorphism group of $( \mathcal{D}; \sqsubseteq)$}

First, we prove Theorem \ref{15886} using Theorem \ref{46031}.

\begin{proof}[of Theorem \ref{15886}] It is clear that the orbits of the automorphism group are finite as an automorphism can only move a digraph inside its level in $(\mathcal{D}, \sqsubseteq)$. 
Let $o(G)$ denote the size of the orbit of the digraph $G$ (which is therefore a positive integer).

We state that it suffices to present a finite set of digraphs such that the only automorphism fixing them all is the identity. 
To prove that, let $\{C_1, \dots, C_k\}$ be such a set and $\varphi$ be an arbitrary automorphism.
Observe that the images of $C_i$ under $\varphi$ determine $\varphi$ completely, or in other words, the only automorphism agreeing with $\varphi$ on $\{C_1, \dots, C_k\}$ is $\varphi$.
Indeed, with the notations
$$S=\{\alpha \in \mathrm{Aut}( \mathcal{D}; \sqsubseteq): \alpha(C_i)=\varphi(C_i), \;\; i=1, \dots, k\},$$
and $S'=\{\alpha \varphi^{-1} : \alpha \in S\}$, $|S|=|S'|$ holds, and $|S'|=1$ for all elements of $S'$ fix all of $\{C_1, \dots, C_k\}$.
The fact that an automorphism is completely determined by its action on $\{C_1, \dots, C_k\}$ means that the automorphism group has at most $o(C_1)\cdot \ldots \cdot o(C_k)$ elements. That proves our statement.

Finally, we claim that $\{C_1, \dots, C_k\}$ of Theorem \ref{46031} suffices for the purpose above, namely the only automorphism fixing them all is the identity. 
Let $\varphi$ be an automorphism that fixes all $C_i$. 
Let $G\in \mathcal{D}$ be arbitrary. 
We need to show that $\varphi(G)=G$. 
We know from Corollary \ref{58562} that there exists a formula $\phi_G(x)$ with one free variable, that defines $G$ in the first-order language of $(\mathcal{D}, \sqsubseteq, C_1, \dots, C_k)$. 
If we change all occurrences of $C_i$ to $\varphi(C_i)$ in $\phi_G(x)$, then we get a formula $\phi_{\varphi(G)}(x)$ defining $\varphi(G)$. 
For $\varphi$ fixes all $C_i$s, $\phi_G(x)=\phi_{\varphi(G)}(x)$, implying $G=\varphi(G)$.
 \end{proof}

In the remaining part of the section, we present the automorphisms that we know of. 
Here, no claim is proven rigorously, they are all rather conjectures. 
Our intention is just to offer some insight on how the author sees the automorphism group at the moment. 

Before the (semi-)precise definition of our automorphisms, we feel it is useful to give a nontechnical glimpse at them.
Automorphisms map digraphs to digraphs in $\mathcal{D}$.
To define an automorphism $\varphi$, we need to tell how to get $\varphi(G)$ from $G$.
All the automorphisms, that we know of at the moment, share a particular characteristic.
They are all, say, {\it local} in the following sense.
Roughly speaking, to get $\varphi(G)$ from $G$, one only needs to consider and modify $G$'s at most two element substructures according to some given rule.

To make this clearer, we give an example. 
Let $\varphi(G)$ be the digraph that we get from $G$ such that we change the direction of the edges on those two element substructures of $G$ that have loops on both vertices.
It is easy to see that this defines an automorphism, indeed.
Perhaps, one would quickly discover the automorphism that gets $\varphi(G)$ by reversing all edges of $G$, but this is different.  
In this example, the modification of $G$ happens only locally, namely on 2-element substructures.
All the automorphisms, that we know of, share this property.

Now, we define some of our automorphisms, $\varphi_i$, (semi-)precisely. 
We tell how to get $\varphi_i(G)$ from $G$. 
One of the most trivial automorphisms is
\begin{itemize}
\item $\varphi_1$: where there is a loop, clear it, and vice versa, to the vertices with no loop, insert one.
\end{itemize}
Observe that this automorphism operates with the 1-element substructures. 
Now we start to make use of the labels of Fig. \ref{41511}.
\begin{itemize}
\item $\varphi_2$: change the substructures (isomorphic to) $E$ to $E'$ and vice versa.
\item $\varphi_3$: change the substructures (isomorphic to) $L$ to $L'$ and vice versa.
\item $\varphi_4$: reverse the edges in the substructures (isomorphic to) $P$.
\item $\varphi_5$: reverse the edges in the substructures (isomorphic to) $Q$.
\end{itemize}
Let $S_4$ denote the symmetric group over the four-element set $\{A,B,C,D\}$, and $\pi \in S_4$.
We define
\begin{itemize}
\item $\varphi_\pi$: We change the substructures (isomorphic to) $X \in \{A,B,C,D\}$ to $\pi(X)$ (such that the loops remain in place).
\end{itemize}
Observe that, with the exception of $\varphi_1$, the automorphisms defined above do not touch loops (when getting $\varphi_i(G)$ from $G$).
We conjecture that these automorphisms generate the whole automorphism group.

Finally, we investigate the structure of the group of our conjecture. 
Let $I$ denote the set of possible indexes of our $\varphi$s, namely
$$I=\{1,\dots,  5\} \cup \{\pi \in S_4\}.$$
Let $\langle \rangle$ stand for subgroup generation.
Let $S=\langle\varphi_i : i\in I\rangle$ denote the group of our conjecture.
It seems that $S$ splits into the internal semidirect product
$$S=\langle\varphi_i : i\in I\setminus\{1\}\rangle \rtimes \langle \varphi_1\rangle.$$
Furthermore, the first factor appears to be a(n internal) direct product
$$\langle\varphi_2 \rangle \times \langle\varphi_3\rangle \times \langle\varphi_4\rangle \times \langle\varphi_5\rangle \times \langle\varphi_\pi : \pi \in S_4\rangle.$$ 
Here, at the last factor, the subgroup generation is just a technicality as, clearly, the $\varphi_\pi$s constitute a subgroup themselves.
These observations all need a proper checking, but they give rise to the conjecture that $S$ is isomorphic to
$$(\mathbb{Z}_2^4 \times S_4)\rtimes_{\alpha} \mathbb{Z}_2,$$
where $S_4$, again, denotes the symmetric group over the set $\{A,B,C,D\}$, and $\alpha$ is the following. 
Obviously, $\alpha(0)=\text{id}\in  \mathrm{Aut}(\mathbb{Z}_2^4 \times S_4)$.
To define $\alpha(1)$, let $p,q,r,s \in \{0,1\}$ and $\pi \in S_4$.
Then
\begin{equation} \label{uzzrhuiwre}
\alpha(1): (p,q,r,s,\pi) \mapsto (q,p,s,r, (BC)\pi(BC)),
\end{equation}
where $(BC)$ is just the usual cycle notation of the permutation of $S_4$ that takes $B$ to $C$ and vice versa.
Note that the group of our conjecture has 768 elements.
Even though we cannot prove that there are no more automorphisms beyond the ones in $S$, we conjecture so.
\begin{nconjecture} The automorphism group of the partial order $(\mathcal{D}; \sqsubseteq)$ is isomorphic to $(\mathbb{Z}_2^4 \times S_4)\rtimes_{\alpha} \mathbb{Z}_2$, with the $\alpha$ defined above (around (\ref{uzzrhuiwre})).
\end{nconjecture}

\section{The proof of Theorem \ref{46031} without some technicalities}

As long and technical as it may seem, the whole proof of Theorem \ref{46031} is based on a simple idea, which we outline here. 
We get substructures of a directed graph by leaving out vertices, while, to get  embeddable digraphs, we can leave out vertices and edges both. 
We want to define the latter, so we need to be able to leave out edges somehow.  
Our main idea is the following. 
In a digraph $G$, if there is an edge $(u, v) \in E(G)$, then we add a vertex and two edges to ``support'' the edge $(u, v)$. 
Namely, we add $w$ to the set of vertices, and the edges $(u, w)$ and $(w, v)$ to the set of edges.
After the addition, we say that the edge $(u, v)$ is ``supported''.
The idea is that the supportedness of an edge can be terminated by leaving out a vertex, in the previous example $w$, what we can do by taking substructures.
Roughly, what we should do is: support all edges, take a substructure, and in one more step, leave only the supported edges in. 
Of course, there seem to be many problems with this (if told in such a simplified way).
Firstly, how can we distinguish between the supporting vertices and the original ones?
This appears to be an essential part of the plan.
Secondly, the plan ended with ``leave only the supported edges in'' which just looks running into the original problem again: we cannot leave edges out.
Even though the plan seems flawed for these reasons, it is manageable. 
The whole section is no more than building the apparatus and carrying it out.

After these paper-specific lines, we talk a little more general. 
There is a method that is often used in the family of papers that we could call ``Definability...'' papers, a family already mentioned in the introduction and encompassing the present paper.
This method is marking the vertices of our structures with distinguishable markers, in this particular paper, large circles.
This is a natural approach that enables us to work with our structures elementwise, and that seems unavoidable.
Referring to the elements piecewise seems only possible if they are properly distinguished. 
Personally, the author saw this approach first in the series of papers \cite{Jezek2009_1, Jezek2010, Jezek2009_3, Jezek2009_4} by Je\v{z}ek and McKenzie, but, being so natural, wouldn't be surprised to learn it emerged earlier in the literature.

\begin{definition}\label{108045}
In this section, we use three particular automorphisms:
\begin{itemize}
\item the {\it loop-exchange automorphism}, denoted by $l$, which is $\varphi_1$ (of the previous section),
\item the {\it edge-reverse (transposition) automorphism}, denoted by $t$, which just reverses all edges in the digraphs, and
\item the {\it complement automorphism}, denoted by $c$, which replaces $E(G)$ with its complement, $V(G)^2 \setminus E(G)$.
\end{itemize}
\end{definition}

Some basic definitions follow 

\begin{definition}\label{349587897819} For digraphs $G, G'\in \mathcal{D}$, let $G\mathrel{\dot{\cup}} G'$ denote their disjoint union, as usual.
\end{definition}

\begin{definition}\label{defE_n} Let ${E_n}$ $(n=1,2,\dots )$ denote the ``empty'' digraph with $n$ vertices and ${F_n}$ $(n=1,2,\dots )$ denote the ``full'' digraph with $n$ vertices: $$V(E_n)=\{v_1,v_2,\dots, v_n\},\;\;  E(E_n)=\emptyset,$$
$$V(F_n)=\{v_1,v_2,\dots, v_n\}, \;\; E(F_n)=V(F_n)^2.$$
\end{definition}

\begin{definition}\label{defIOL}
Let $I_n$ $(n=1,2,\dots)$, $O_n$ $(n=3,4,\dots)$, and $L_n$ $(n=1,2,\dots)$ be the following (Fig. \ref{O_n abra}.) digraphs: 
\begin{displaymath}
V(I_n)=V(O_n)=V(L_n)=\{v_1,v_2,\dots, v_n\}, 
\end{displaymath}
\begin{displaymath}
E(I_n)=\{(v_1,v_2),(v_2,v_3),\dots, (v_{n-1}, v_n) \},
\end{displaymath}
\begin{displaymath}
E(O_n)=\{(v_1,v_2),(v_2,v_3),\dots, (v_{n-1}, v_n), (v_n, v_1)\},
\end{displaymath}
\begin{displaymath}
E(L_n)=\{(v_1,v_1),(v_2,v_2),\dots, (v_n, v_n)\}.
\end{displaymath}
The digraphs $I_n$ are called {\it lines}, and the digraphs $O_n$ are called {\it circles}.
\end{definition}

Note $E_1 =I_1$.

\begin{figure}[h]
\begin{center}
\begin{tikzpicture}[line cap=round,line join=round,>=triangle 45,x=0.6cm,y=0.6cm]
\clip(2.22,0.5) rectangle (13.1,5.88);
\draw [->] (3.38,1.82) -- (3.4,2.68);
\draw [->] (3.4,2.68) -- (3.38,3.58);
\draw [->] (3.38,3.58) -- (3.36,4.54);
\draw [->] (3.36,4.54) -- (3.34,5.42);
\draw [->] (5.96,1.96) -- (5.36,3.54);
\draw [->] (5.36,3.54) -- (6,5);
\draw [->] (6,5) -- (7.8,4.98);
\draw [->] (7.8,4.98) -- (8.52,3.58);
\draw [->] (8.52,3.58) -- (8,2);
\draw [->] (8,2) -- (5.96,1.96);
\draw (3.18,1.28) node[anchor=north west] {$I_5$};
\draw (6.65,1.27) node[anchor=north west] {$O_6$};
\draw (11.1,1.3) node[anchor=north west] {$L_6$};
\begin{scriptsize}
\fill [color=black] (3.38,1.82) circle (1.5pt);
\draw[color=black] (3.52,2.1);
\fill [color=black] (3.4,2.68) circle (1.5pt);
\draw[color=black] (3.56,2.96);
\fill [color=black] (3.38,3.58) circle (1.5pt);
\draw[color=black] (3.54,3.86);
\fill [color=black] (3.36,4.54) circle (1.5pt);
\draw[color=black] (3.52,4.82);
\fill [color=black] (3.34,5.42) circle (1.5pt);
\draw[color=black] (3.5,5.7);
\fill [color=black] (5.96,1.96) circle (1.5pt);
\draw[color=black] (6.1,2.24);
\fill [color=black] (5.36,3.54) circle (1.5pt);
\draw[color=black] (5.52,3.82);
\fill [color=black] (6,5) circle (1.5pt);
\draw[color=black] (6.16,5.28);
\fill [color=black] (7.8,4.98) circle (1.5pt);
\draw[color=black] (7.9,5.26);
\fill [color=black] (8.52,3.58) circle (1.5pt);
\draw[color=black] (8.66,3.86);
\fill [color=black] (8,2) circle (1.5pt);
\draw[color=black] (8.16,2.28);
\fill [color=black] (10.56,2.02) circle (1.5pt);
\draw[color=black] (10.7,2.3);
\fill [color=black] (12.1,2.04) circle (1.5pt);
\draw[color=black] (12.26,2.32);
\fill [color=black] (10.56,3.48) circle (1.5pt);
\draw[color=black] (10.72,3.76);
\fill [color=black] (12.2,3.5) circle (1.5pt);
\draw[color=black] (12.36,3.78);
\fill [color=black] (10.52,4.78) circle (1.5pt);
\draw[color=black] (10.68,5.06);
\fill [color=black] (12.14,4.82) circle (1.5pt);
\draw[color=black] (12.3,5.1);
\draw[rotate around={-90:(10.56,2.02)}] [->] (10.56,2.02) arc (360:10:5pt);
\draw[rotate around={-90:(12.1,2.04)}] [->] (12.1,2.04) arc (360:10:5pt);
\draw[rotate around={-90:(10.56,3.48)}] [->] (10.56,3.48) arc (360:10:5pt);
\draw[rotate around={-90:(12.2,3.5)}] [->] (12.2,3.5) arc (360:10:5pt);
\draw[rotate around={-90:(10.52,4.78)}] [->] (10.52,4.78) arc (360:10:5pt);
\draw[rotate around={-90:(12.14,4.82)}] [->] (12.14,4.82) arc (360:10:5pt);
\end{scriptsize}
\end{tikzpicture}
\caption{}
\label{O_n abra}
\end{center}
\end{figure}

\begin{definition} A directed graph is called an {\it $IO$-graph} if it satisfies the following conditions. 
The only one-element substructure of it is $E_1$. 
If $X$ is a two-element substructure then it is either $E_2$ or $I_2$. 
If $X$ is a three-element substructure then $X$ is $E_3$, or $I_2 \dcup E_1$, or $I_3$, or $O_3$. 
Let the  set of $IO$-graphs be denoted by $IO$.
\end{definition}

\begin{nlemma} The set $IO$ is definable.
\end{nlemma}

\begin{proof} Observe that the set $IO$ is already given by a first-order definition, using the one, two, and three element digraphs as constants.
 \end{proof}

Observe that the set $IO$ is closed under taking substructures. The following lemma motivates our notation $IO$.

\begin{nlemma} A directed graph is an IO-graph if and only if it is a disjoint union of lines and/or circles.
\end{nlemma}

\begin{proof} Straightforward induction on the number of vertices suffices, using the closedness mentioned prior to the lemma.
 \end{proof}

\begin{nlemma} The set $\{ O_n: n\geq 3\}$ is definable.
\end{nlemma}

\begin{proof} It is clear that all elements of the set are $IO$-graphs, we just need to choose which. It is easy to see that, in $IO$, those that have a unique lower-cover (within $IO$) are:
$$\underbrace{G\dcup \dots \dcup G}_{\text{$k$ copies}}\text{, where }G\in \{E_1, I_2\}\cup \{O_n:n\geq 3\},$$
for $k\geq 1$ except when $X=E_1$, then $k>1$. In this set, the desired digraphs are exactly those that are minimal (in this particular set) and have $I_3$ or $O_3$ as a substructure. 
 \end{proof}

\begin{definition} A digraph is called {\it loop-full} if all vertices have loops on them, and {\it loop-free} if none. The {\it loop-full part} of a digraph is the maximal loop-full substructure of it, and the {\it loop-free part} is the maximal loop-free substructure.
\end{definition}

\begin{nlemma}\label{fkuerz8w9} The relation
$$\{(G, F, G\dcup F): G, F\in\mathcal{D},\text{ $G$ is loop-full and $F$ is loop-free}\}$$
is definable.
\end{nlemma}

\begin{proof} The relation consists of those triples $(X,Y,Z)$ for which
\begin{itemize}
\item $X$ is the loop-full part of $Z$, 
\item $Y$ is the loop-free part of $Z$, and
\item there is no two element substructure of $Z$ that consists exactly one loop and has a non-loop edge in it.
\end{itemize}
 \end{proof}

\begin{definition} Let $L_{\to}$ denote the digraph with $$V(L_{\to})=\{v_1, v_2\},\text{ and }E(G)=\{(v_1, v_1), (v_1, v_2)\}.$$
\end{definition}

\begin{definition}\label{060639} Let $G$ be a loop-full digraph with $V(G)=\{v_1, \dots, v_n\}$. Then $l(G)$ is loop-free. Let the set of its vertices be $l(G)=\{v'_1, \dots, v'_n\}$ (to recall $l$, see Def. \ref{108045}) with
$$\text{for }i\neq j:\; (v'_i,v'_j)\in E(l(G)) \Leftrightarrow (v_i,v_j)\in E(G).$$
Let $G\to l(G)$ denote the digraph for which
$$V(G\to l(G))=V(G) \cup V(l(G)), \text{ and}$$
$$E(G\to l(G))=E(G)\cup E(l(G)) \cup \{(v_i, v'_i): 1\leq i \leq n\}.$$
\end{definition}

\begin{nlemma}\label{vxjfhdsuhiz} The relation
$$\{(G, l(G), G\to l(G)): G\in\mathcal{D},\text{ $G$ is loop-full}\}$$
is definable.
\end{nlemma}

\begin{proof}
Let us consider the triples $(X,Y,Z)$ for which
\begin{itemize}
\item $X$ is the loop-full part of $Z$, and $Y$ is the loop-free part of $Z$, 
\item $X\dcup E_1 \nsqb Z$, and $Y\dcup L_1 \nsqb Z$ (both are definable by Lemma \ref{fkuerz8w9}), 
\item on two points, the only substructure having exactly one loop and at least one non-loop edge is $L_{\to}$, and
\item no digraph of the first two pictures of Fig. \ref{xcljsdisro} is a substructure. We consider the dashed edges possibilities, either we draw them (individually) or not. In this way, there are $6$ (isomorphism types) encoded into the first two pictures of Fig. \ref{xcljsdisro}. We exclude them all.
\end{itemize}
Now we have ensured that the edges $L_{\to}$ constitute a bijection between the vertices of $X$ and $Y$ in $Z$. It only remains to force this bijection to be edge and non-edge preserving as well. This can be done by requiring the additional the property
\begin{itemize}
\item Consider the third picture of Figure \ref{xcljsdisro} as before, the dashed edges are possibilities. We forbid those from being substructures in which the dashed edges are not symmetrically drawn on the two (loop-full and loop-free) sides.
\end{itemize}
 \end{proof}

\begin{figure}[h]
\begin{center}
\begin{tikzpicture}[line cap=round,line join=round,>=triangle 45,x=0.71cm,y=0.71cm]
\clip(0.5,0.5) rectangle (10,3.5);

\fill [color=black] (1,1) circle (1.5pt);
\draw[rotate around={90:(1,1)}] [->] (1,1) arc (360:10:5pt);
\fill [color=black] (1,3) circle (1.5pt);
\draw[rotate around={-90:(1,3)}] [->] (1,3) arc (360:10:5pt);
\fill [color=black] (2,2) circle (1.5pt);

\draw [->,dashed]  (1,1) .. controls(0.5,2) ..  (1,3);
\draw [->,dashed]  (1,3) .. controls(1.5,2) ..  (1,1);
\draw [->]  (1,1) --  (2,2);
\draw [->]  (1,3) --  (2,2);

\fill [color=black] (4,2) circle (1.5pt);
\draw[rotate around={0:(4,2)}] [->] (4,2) arc (360:10:5pt);
\fill [color=black] (5,1) circle (1.5pt);
\fill [color=black] (5,3) circle (1.5pt);

\draw [->,dashed]  (5,1) .. controls(4.5,2) ..  (5,3);
\draw [->,dashed]  (5,3) .. controls(5.5,2) ..  (5,1);
\draw [->]  (4,2) --  (5,1);
\draw [->]  (4,2) --  (5,3);

\fill [color=black] (7,1) circle (1.5pt);
\draw[rotate around={80:(7,1)}] [->] (7,1) arc (360:10:5pt);
\fill [color=black] (7,3) circle (1.5pt);
\draw[rotate around={-80:(7,3)}] [->] (7,3) arc (360:10:5pt);

\fill [color=black] (9,1) circle (1.5pt);
\fill [color=black] (9,3) circle (1.5pt);

\draw [->,dashed]  (7,1) .. controls(6.5,2) ..  (7,3);
\draw [->,dashed]  (7,3) .. controls(7.5,2) ..  (7,1);
\draw [->,dashed]  (9,1) .. controls(8.5,2) ..  (9,3);
\draw [->,dashed]  (9,3) .. controls(9.5,2) ..  (9,1);
\draw [->]  (7,1) --  (9,1);
\draw [->]  (7,3) --  (9,3);

\end{tikzpicture}
\caption{}
\label{xcljsdisro}
\end{center}
\end{figure}

We are going to need some basic arithmetic later. We define addition in the following lemma.

\begin{nlemma}\label{cxnsdf9ds} The following relation is definable:
$$\{(E_n,E_m,E_{n+m}): n,m\geq 1\}.$$
\end{nlemma}

\begin{proof} The set $\{E_n\}$ is definable as it consists of $E_1$ plus those digraphs which have only $E_2$ as a two-element substructure.
$E_n \dcup (L_m \to E_m)$ is the digraph $X$ for which
\begin{itemize}
\item $E_n \dcup L_m \sqsubseteq X$ (using Lemmas \ref{fkuerz8w9} and \ref{vxjfhdsuhiz}),
\item $L_m \to E_m \sqsubseteq X$ (using Lemma \ref{vxjfhdsuhiz}),
\item the second digraph of Fig. \ref{xcljsdisro}, without the dashed edges, is not a substructure,
\item $E_{n+1} \dcup L_m \nsqb X$ ($E_{n+1}$ is just the cover of $E_n$ in $\{E_n\}$),
\item on two vertices, the only substructure having a non-loop edge is $L_{\to}$,
\item the maximal loop-full substructure of $X$ is $L_m$, and
\item the maximal loop-free substructure of $X$ is of the form $E_i$.
\end{itemize}
The $E_i$ of the last condition is $E_{n+m}$.
 \end{proof}

\begin{nlemma}\label{jbnhbcez} The following relation is definable:
\begin{equation}\label{20893}
\{(G,F): \text{$G$ and $F$ have the same number of vertices}\}.
\end{equation}
\end{nlemma}

\begin{proof} We ``determine'' the number of vertices for the loop-full and the loop-free parts of the graphs separately and add them using Lemma \ref{cxnsdf9ds}.  Let $G_1$ denote the loop-full part of $G$, and $G_2$ denote the loop-free part. Let $X$ denote the digraph with the following properties:
\begin{itemize}
\item The loop-full part of $X$ is $G_1$, and the loop-free part is $E_i$ for some $i$.
\item On two points, the only substructure having exactly one loop and at least one non-loop edge is $L_{\to}$.
\item $G_1 \dcup E_1 \nsqb X$, and $E_i \dcup L_1 \nsqb X$.
\item Just as in the proof of Lemma \ref{vxjfhdsuhiz}, no digraph of the 6 digraphs of the first two pictures of Fig. \ref{xcljsdisro} is a substructure. (No matter, we wouldn't even need all 6 in this case.)
\end{itemize}

Observe that in $X$, the edges $L_{\to}$ constitute a bijection between $G_1$ and $E_i$, consequently $i$ in the first condition is $|V(G_1)|$. 

Now we proceed analogously for  the loop-free part, $G_2$. 
We do not write all the conditions down again, as they are just the ones above converted with the automorphism $l$. 
This way, we get $L_j$ with $j=|V(G_2)|$. 
We already have $E_i$ and $L_j$ defined, such that $i+j=|V(G)|$. 
To conclude, we use the relation of Lemma \ref{vxjfhdsuhiz} to get $E_j$ and Lemma \ref{cxnsdf9ds} to obtain the desired $E_{i+j}$, marking the number of vertices of $G$. 

Finally, $(G,F)\in \eqref{20893}$ holds if and only if, by doing the same, we get the same $E_{i'+j'}$ marking the number of vertices.
 \end{proof}

We define some more arithmetic in the following lemma, namely multiplication.

\begin{nlemma}\label{gvtspjhj} The following relation is definable:
$$\{(E_n,E_m, E_{nm}) : n,m\geq 1\}.$$
\end{nlemma}

\begin{proof} The relation $\{(E_i, F_i): i=1,2, \dots\}$ is definable as, beyond $(E_1, F_1)$, for $i>1$, $F_i$ is the only digraph having the same vertices as $E_i$ that has only $F_2$ as a two element substructure. 
Let $X$ be a digraph that is maximal with the following properties:
\begin{enumerate}
\item $E_1 \nsqb X$ to ensure that the relation $E(X)$ is reflexive.
\item $l(I_2) \nsqb X$ to ensure that the relation $E(X)$ is symmetric.
\item The digraph of Fig. \ref{nvxcpi8ed74w} is not a substructure of $X$ to ensure that the relation $E(X)$ is transitive.
\item $L_n$ is the maximal $L_i$ substructure.
\item $F_m$ is the maximal $F_i$ substructure.
\end{enumerate}

The conditions 1-3 force $E(X)$ to be an equivalence. Condition 4 tells the equivalence has at most $n$ classes and condition 5 requires the classes to have at most $m$ elements. 
It is easy to see that such an equivalence relation has a base set of at most $nm$ elements, hence $|V(X)|=nm$.
Thus, using Lemma \ref{jbnhbcez}, we are done.
 \end{proof}

\begin{figure}[h]
\begin{center}
\begin{tikzpicture}[line cap=round,line join=round,>=triangle 45,x=1.0cm,y=1.0cm]
\clip(0.5,0.5) rectangle (5.5,1.5);

\fill [color=black] (1,1) circle (1.5pt);
\draw[rotate around={0:(1,1)}] [->] (1,1) arc (360:10:5pt);
\fill [color=black] (3,1) circle (1.5pt);
\draw[rotate around={-90:(3,1)}] [->] (3,1) arc (360:10:5pt);
\fill [color=black] (5,1) circle (1.5pt);
\draw[rotate around={-180:(5,1)}] [->] (5,1) arc (360:10:5pt);

\draw [->]  (1,1) ..controls(2,1.5)..  (3,1);
\draw [->]  (3,1) ..controls(2,0.5)..  (1,1);
\draw [->]  (3,1) ..controls(4,1.5)..  (5,1);
\draw [->]  (5,1) ..controls(4,0.5)..  (3,1);
\end{tikzpicture}
\caption{}
\label{nvxcpi8ed74w}
\end{center}
\end{figure}

\begin{nlemma}\label{gvtspjhj} Disjoint union of $IO$ graphs is definable, i.e. the following relation is definable:
$$\{(G_1,G_2, G_1 \dcup G_2) : G_1, G_2 \in IO\}.$$
\end{nlemma}

\begin{proof}
Though the notion $l(G)\to G$ was not defined as is in Definition \ref{060639}, it is such an analogue that the reader surely can decipher without effort.
Using $G_1$ and $G_2$, we want to define 
\begin{equation}\label{52540}
G_1 \dcup (l(G_2)\to G_2),
\end{equation}
whose loop-free part is the sought $G_1 \dcup G_2$.
For this, let $X$ satisfy the following conditions.
\begin{itemize}
\item $|V(X)|=|V(G_1)|+2|V(G_2)|$ (using Lemmas \ref{jbnhbcez} and \ref{cxnsdf9ds}),
\item $G_1 \dcup l(G_2) \sqsubseteq  X$ (using Lemma \ref{fkuerz8w9}),
\item $l(G_2)\to G_2 \sqsubseteq X$ (using Lemma \ref{vxjfhdsuhiz}), and
\item $l(F_2) \not\sqsubseteq X$.
\end{itemize}
It easy to see that these three conditions ensure that \eqref{52540} is embeddable (not substructure!) into $X$: there can be edges between the substructures $G_1$ and $G_2$ which we need to exclude. 
If there is an edge from $G_2$ to $G_1$ (in this particular direction), then the first graph of Fig.  \ref{4574234966} is a substructure, without the dashed edges. 
Analogously, if an edge goes from $G_1$ to $G_2$, then the second digraph of Fig.  \ref{4574234966} is a substructure, without the dashed edges.
Edges going both directions is forbidden by the last condition above.
Thus we need to exclude these two substructures. Let $Y$ satisfy the following conditions.
\begin{itemize}
\item $|V(Y)|=|V(G_2)|+2$, and $Y \sqsupseteq l(G_2)$.
\item $I_2$ and $L_{\to}$ are substructures of $Y$.
\item The digraph of Fig. \ref{82846} is not a substructure of $Y$.
\end{itemize}

These three conditions do not define the two digraphs of Fig.  \ref{4574234966} without the dashed edges, they rather define the set of those with the dashed edges meant as possibilities, as usual. 
However none of the dashed edges can actually appear in our $X$ so by excluding all such, we do not do more than by excluding only the two without the dashed edges. 
Finally, \eqref{52540} is the loop-free part of $X$ . 
 \end{proof}

\begin{figure}[h]
\begin{center}
\begin{tikzpicture}[line cap=round,line join=round,>=triangle 45,x=0.8cm,y=0.8cm]
\clip(-1,0) rectangle (9.2,4);

\Large

\draw plot [smooth cycle] coordinates { (0,0) (-1,2) (0,4) (1,2)};
\draw (0,2) node[anchor=center] {$l(G_2)$};
\fill [color=black] (0,3) circle (1.5pt);
\draw[rotate around={0:(0,3)}] [->] (0,3) arc (360:10:5pt);
\fill [color=black] (2,2) circle (1.5pt);
\fill [color=black] (3,1) circle (1.5pt);
\fill [color=black] (0,1) circle (1.5pt);
\draw[rotate around={0:(0,1)}] [->] (0,1) arc (360:10:5pt);
\fill [color=black] (0,.5) circle (1.5pt);
\draw[rotate around={0:(0,.5)}] [->] (0,.5) arc (360:10:5pt);

\draw [->]  (0,3) --  (2,2);
\draw [->, thick]  (2,2) --  (3,1);
\draw [->, dashed]  (3,1) ..controls(2,3)..  (0,3);
\draw [->, dashed]  (0,3) ..controls(2,3.5)..  (3,1);
\draw [->, dashed]  (3,1) --  (0,.5);
\draw [->, dashed]  (2,2) ..controls(1,1.75)..  (0,1);
\draw [->, dashed]  (0,1) ..controls(1,1.25)..  (2,2);


\draw plot [smooth cycle] coordinates { (6,0) (5,2) (6,4) (7,2)};
\draw (6,2) node[anchor=center] {$l(G_2)$};
\fill [color=black] (6,3) circle (1.5pt);
\draw[rotate around={0:(6,3)}] [->] (6,3) arc (360:10:5pt);
\fill [color=black] (8,2) circle (1.5pt);
\fill [color=black] (9,1) circle (1.5pt);
\fill [color=black] (6,1) circle (1.5pt);
\draw[rotate around={0:(6,1)}] [->] (6,1) arc (360:10:5pt);
\fill [color=black] (6,.5) circle (1.5pt);
\draw[rotate around={0:(6,.5)}] [->] (6,.5) arc (360:10:5pt);

\draw [->]  (6,3) --  (8,2);
\draw [->, thick]  (9,1) --  (8,2);
\draw [->, dashed]  (9,1) ..controls(8,3)..  (6,3);
\draw [->, dashed]  (6,3) ..controls(8,3.5)..  (9,1);
\draw [->, dashed]  (9,1) --  (6,.5);
\draw [->, dashed]  (8,2) ..controls(7,1.75)..  (6,1);
\draw [->, dashed]  (6,1) ..controls(7,1.25)..  (8,2);

\end{tikzpicture}
\caption{}
\label{4574234966}
\end{center}
\end{figure}

\begin{figure}[h]
\begin{center}
\begin{tikzpicture}[line cap=round,line join=round,>=triangle 45,x=.7cm,y=0.7cm]
\clip(0.5,0.5) rectangle (3.5,3.5);

\fill [color=black] (1,1) circle (1.5pt);
\draw[rotate around={80:(1,1)}] [->] (1,1) arc (360:10:5pt);
\fill [color=black] (1,3) circle (1.5pt);
\draw[rotate around={-80:(1,3)}] [->] (1,3) arc (360:10:5pt);

\fill [color=black] (3,1) circle (1.5pt);
\fill [color=black] (3,3) circle (1.5pt);

\draw [->]  (1,1) --  (1,3);
\draw [->]  (3,1) --  (3,3);
\draw [->]  (1,1) --  (3,1);
\draw [->]  (1,3) --  (3,3);
\end{tikzpicture}
\caption{}
\label{82846}
\end{center}
\end{figure}

\begin{nlemma}\label{42847} The following set is definable.
\begin{equation}\label{9859323333}
\{G:\text{ $G$ is a disjoint union of circles of different sizes}\}.
\end{equation}
\end{nlemma}

\begin{proof}
The set of digraphs that are disjoint unions of circles contains those $IO$ graphs that have unique upper-covers (in the set IO). In this set, the digraphs of the form $O_i \dcup O_i$ are those that have a unique circle substructure $O_i$ and have twice as many vertices as $O_i$. We have defined two sets of digraphs, the set of the lemma is just the set of those digraphs of the first set that have no substructures from the second.
 \end{proof}

\begin{nlemma}\label{74054} The following relation is definable.
\begin{equation}\label{59943}
\begin{split}
\{(O^*, G\dcup O^*): &\text{$G\in \mathcal{D}$  and $O^*$ is a disjoint union of $|V(G)|$-many circles of} \\
&\text{different sizes such that the  smallest has at least $|V(G)|+1$ vertices}\}.
\end{split}
\end{equation}
\end{nlemma}

\begin{proof}
First, we define a relation counting the number of circles in $O^*$, actually we formulate it without the restriction on the sizes of the circles:
\begin{equation}\label{85796}
\{(E_i, O):\text{$O$ is a disjoint union of $i$ circles}\}. 
\end{equation}
The set of $O$'s of this relation was defined in the first sentence of the proof of Lemma \ref{42847}. Let $O'$ denote such a substructure of $O$ that has no circle in it and has a maximal number of vertices with this property. 
Then $i+|V(O')|=|V(O)|$ holds for the $i$ of \eqref{85796}, thus we can conclude with the addition relation defined earlier.

Let $O^*$ be an element of the set defined in Lemma \ref{42847} and $i$ be the number of its circles. Let $X$ satisfy:
\begin{itemize}
\item $|V(X)|=|V(O^*)|+i$.
\item The smallest circle in $O^*$ has at least $i+1$ vertices.
\item $O^* \sqsubseteq X$.
\item $X$ does not have a substructure $Y$ for which
\begin{itemize}
\item $|V(Y)|=|V(O^*)|+1$, and $Y \sqsupseteq O^*$,
\item $Y$ is loop-free, and 
\item $Y$ is not an $IO$-graph.
\end{itemize}
\item $X$ does not have a substructure $Y$ for which
\begin{itemize}
\item $|V(Y)|=|V(O^*)|+1$, and $Y \sqsupseteq O^*$,
\item $Y$ has a loop in it, and
\item $Y$ has one of $L_{\to}$ or $t(L_{\to})$ or $c(L_1 \dcup E_1)$ as a substructure.
\end{itemize}
\end{itemize}

With these properties, $X$ is of the required form $G\dcup O^*$.
 \end{proof}

We use the term {\it weakly connected} in the usual sense, i. e. disregarding the direction of the edges. 
Our digraphs split into {\it weakly connected components}, naturally.
Let us use the abbreviation {\it wcc}=``weakly connected component'' and {\it wccs} for the plural.

\begin{definition} Let $O^*$ be a digraph that is a disjoint union of circles of different sizes, as usual, and let $G$ be an arbitrary digraph.
We introduce the notation $G_N^{O^*}$ for the digraph that is the disjoint union of weakly connected components of $G$ not embeddable into $O^*$.
\end{definition}

\begin{nlemma}\label{608589} The following relation is definable.
$$\{(O^*, G, G_N^{O^*}) : G\in \mathcal{D}, \text{ and $O^*$ is a disjoint union of circles of different sizes}\}$$
\end{nlemma}

\begin{proof} The following conditions suffice.
\begin{itemize}
\item $O^* \in \text{(\ref{9859323333})}$,
\item $G_N^{O^*} \sqsubseteq G$,
\item for all $X\dcup O^*$ (using \eqref{59943}), $G_N^{O^*} \sqsubseteq X\dcup O^*$ implies $|V(X)| \geq |V(G_N^{O^*})|$, and
\item $G_N^{O^*}$ is maximal with the properties above.
\end{itemize}
The third condition forces $G_N^{O^*}$ to have only such wccs that are not substructures of $O^*$.
Note that the third condition can be encoded into a first-order formula using Lemmas \ref{cxnsdf9ds} and  \ref{jbnhbcez}.
 \end{proof}

\begin{nlemma} The following relation is definable.
\begin{equation} \label{bade754}
\{(O^*, G, G\dcup O^*): (O^*, G\dcup O^*) \in \eqref{59943}\}
\end{equation}
\end{nlemma}

\begin{proof} The relation in question is the set of triples $(O^*, G, X)$ for which
\begin{enumerate}
\item\label{88643} $(O^*, X) \in \text{\eqref{59943}}$,
\item\label{265325} $|V(X)|=|V(O^*)|+|V(G)|$,
\item \label{ueiowru} $X_N^{O^*}=G_N^{O^*}$ (using Lemma \ref{608589}), and
\item\label{ydi7r} $Y \dcup O^* \sqsubseteq X$, for all IO-substructures $Y$ of $G$.
\end{enumerate}

Condition \ref{88643} entails $X=G'\dcup O^*$ for some $G'$. With this notation, what we have to prove is $G=G'$. 

Condition \ref{265325} adjusts the size of $G'$.

Before explaining Conditions \ref{ueiowru} and \ref{ydi7r}, we introduce a notation needed.
For a digraph $G$, its wccs fall into two categories, those that are IO-graphs, and those that are not.
Let $G_{\mathrm{IO}}$ and $G_{\mathrm{NIO}}$ stand for the substructures of $G$ consisting of the IO- and non-IO-wccs of $G$, respectively. 
To prove $G=G'$, it is enough to show that $G_{\mathrm{IO}}=G'_{\mathrm{IO}}$ and $G_{\mathrm{NIO}}=G'_{\mathrm{NIO}}$ both hold.

Condition \ref{ueiowru} forces $G_{\mathrm{NIO}}=G'_{\mathrm{NIO}}$, because we clearly have $X_N^{O^*} \supseteq G'_{\mathrm{NIO}}$ (observe that this is not substructureness, but containment between parts of digraphs).

Condition \ref{ydi7r} is to ensure $G_{\mathrm{IO}}=G'_{\mathrm{IO}}$, but seeing it serves its purpose is far from trivial.
Our goal is to prove that for any wcc $W$ of $G_{\mathrm{IO}}$: 
\begin{equation} \label{jhfgdgasg}
\begin{split}
&\text{the number of wccs isomorphic to $W$ in $G'_{\mathrm{IO}}$ is greater } \\
&\text{or equal to the corresponding number of $G_{\mathrm{IO}}$.}
\end{split}
\end{equation} 
If this goal is achieved, we are done as we already have $|V(G_{\mathrm{IO}})|=|V(G'_{\mathrm{IO}})|$.
Let us fix an arbitrary wcc $W$ of $G_{\mathrm{IO}}$.
We pick a particular $Y=Y^W$ from the many possible $Y$s of Condition \ref{ydi7r}. 
Firstly, let $Y^W$ have a maximal number of vertices (amongst the IO-substructures of $G$), and on top of that, let it have a maximal number of wccs isomorphic to $W$ in it. 
Note that $Y^W$ need not be uniquely determined by these properties. 
In case it is not, we pick arbitrarily from the ones meeting the requirements.
Condition \ref{ydi7r} grants $Y^W \dcup O^* \sqsubseteq X$.
Let $\varphi$ be a map from $Y^W \dcup O^*$ to $X$ that certifies this fact.
The maximality of the number of vertices of $Y^W$ gives us $G'_{\mathrm{IO}} \subseteq \mathrm{Range}(\varphi)$, and it is easy to see that even if wccs isomorphic to $W$ map into $G'_{\mathrm{NIO}}$, the second maximality property of $Y^W$ implies (\ref{jhfgdgasg}).
 \end{proof}

Some technical tools follow. 
We introduce digraphs that we denote using the symbol $\male$. 
The motivation is the shape of the digraphs, as usual. 
Note, that the same notations were used in the papers \cite{Kunos2015, KunosII} in a slightly different way.

\begin{definition} Let $V(O_n)=\{v_1, \dots, v_n\}$ and let us define two digraphs with
$$V(\male_n):=V(O_n)\cup \{u_1, u_2\}, \; E(\male_n):=E(O_n)\cup \{(v_1, u_1), (u_1, u_2)\}\text{, and} $$
$$V(\male_n^L):= V(\male_n),  \; E(\male_n^L):=E(\male_n) \cup \{(u_2, u_2)\}.$$
Now let $m$ be a different positive integer from $n$ and define $\male_m$ and $\male_m^L$ analogously with $V(\male_m)=V(\male_m^L)=\{v'_1, \dots, v'_m, u'_1, u'_2\}$. 

Now we are going to deal with pairs of the digraphs just defined, which leaves us $4=2 \times 2$ cases with respect to the presence of the loops. To avoid the tiresomeness of listing all 4 possibilities all the time, we resort to the following notation. We say, let $(\Box, \triangledown) \in \{\emptyset, L\}^2$, and for example, in the case $(\Box, \triangledown)=(L, \emptyset)$, we mean $(\male_n^{L}, \male_m)$ by $(\male_n^{\Box}, \male_m^{\triangledown})$, naturally.

Let $(\Box, \triangledown) \in \{\emptyset, L\}^2$. We introduce two more types of digraphs with 
$$V(\male_n^{\Box} \to \male_m^{\triangledown}):=V(\male_n) \cup V(\male_m), \; E(\male_n^{\Box} \to \male_m^{\triangledown}):=E(\male_n^{\Box}) \cup E(\male_m^{\triangledown}) \cup \{(u_2, u'_2)\},\text{and}$$
$$V(\male_n^{\Box} \leftrightarrow \male_m^{\triangledown}):=V(\male_n) \cup V(\male_m), \;  E(\male_n^{\Box} \leftrightarrow \male_m^{\triangledown}):= E(\male_n^{\Box} \to \male_m^{\triangledown}) \cup \{(u'_2, u_2)\}.$$
\end{definition}

\begin{nlemma}\label{25451} The following relation is definable for all $(\Box, \triangledown) \in \{\emptyset, L\}^2$.
\begin{equation}\label{}
\{(E_i, E_j, 
\male_i^{\Box},
\male_i^{\Box} \dcup \male_j^{\triangledown},
\male_i^{\Box} \to \male_j^{\triangledown}, 
\male_i^{\Box} \leftrightarrow \male_j^{\triangledown}): 
i,j >3, i \neq j\}
\end{equation}
\end{nlemma}

The proof is put in the last section for its technical nature.

The following definition introduces a construction of great importance in the remaining half of the proof.

\begin{definition} Let $G$ be a digraph on $n$ vertices with $V(G)=\{v_1, \dots, v_n\}$, and let $(O^*, G\dcup O^*) \in \eqref{59943}$ with $V(O^*)=\{u_i^j: 1\leq j \leq n, 1\leq i \leq k_j\}$ such that the $j$th  circle $O_{k_j}$ of $O^*$ consists of the vertices $\{u_i^j: 1\leq i \leq k_j\}$. 
Let $C(O^*)=\{O_{k_1}, \dots, O_{k_n}\}$ denote the set of the circles of $O^*$ and let $\alpha: C(O^*) \to V(G)$ be a bijective map. We introduce the notation $G\stackrel{\alpha}{\leftarrow} O^*$ for the digraph with
\[
\begin{split}
&V(G\stackrel{\alpha}{\leftarrow} O^*)=V(G\dcup O^*) \cup \{w_1, \dots, w_n\},\; \text{ and} \\
& E(G\stackrel{\alpha}{\leftarrow} O^*)=E(G\dcup O^*) \cup \{(u_1^j, w_j): 1 \leq j \leq n\} \cup \{(w_j, \alpha(O_{k_j})): 1 \leq j \leq n\}.
\end{split}
\]
\end{definition}

\begin{nlemma}\label{26231} The following relation is definable.
\begin{equation}\label{43959}
\{(O^*,G, G\dcup O^*,G\stackrel{\alpha}{\leftarrow} O^*): (O^*,G\dcup O^*)\in \eqref{59943},\; \alpha: C(O^*) \to V(G)\}.
\end{equation}
\end{nlemma}

\begin{proof} As we already defined \eqref{bade754}, we only need to define the digraphs $G\stackrel{\alpha}{\leftarrow} O^*$ (using $O^*$, $G$ and $G\dcup O^*$). The relation of the lemma consists of those triples $(O^*,G,G\dcup O^*,X)$ for which:
\begin{itemize}
\item $|V(X)|=|V(G\dcup O^*)|+|V(G)|$,
\item $G\dcup O^* \sqsubseteq X$,
\item $O_i \sqsubseteq O^*$ implies $\male_i \sqsubseteq X$ or $\male_i^L\sqsubseteq X$,
\item $O_i, O_j \sqsubseteq O^*$ ($i\neq j$) implies 
$\male_i^{\Box} \dcup \male_j^{\triangledown} \sqsubseteq X$, or 
$\male_i^{\Box} \to \male_j^{\triangledown} \sqsubseteq X$, or 
$\male_i^{\Box} \leftrightarrow \male_j^{\triangledown} \sqsubseteq X$ 
for some $(\Box, \triangledown) \in \{\emptyset, L\}^2$.
\end{itemize}

\noindent Unfortunately, these conditions do not ensure $X=G\stackrel{\alpha}{\leftarrow} O^*$ yet.
That is because, say, the tails of the $\male$s can still get entangled. 
Additional technical conditions have to be added to avoid this unwanted scenario.
For its technical nature, this argument is put in the last section.
 \end{proof}

In the following definition we introduce the soul of our proof: the edge-supporting construction. 
Before starting to study the long definition, it is worth to read the simplified idea of it, back at the beginning of this section.

\begin{definition}\label{71859} In this definition, we introduce the {\it edge-supporting construction}. 
Let $G$ be a digraph with
$$V(G)=\{v_1, \dots, v_n\},\text{ and }E(G)=\{e_1, \dots, e_r\}.$$
Note that $r\leq n^2$ is necessary. Let $p_1$ and $p_2$ be two maps from $E(G)$ to $\{v_1, \dots, v_n\}$ defined by the rule
$$\forall e\in E(G): e=(v_{p_1(e)}, v_{p_2(e)}).$$
Let us introduce a digraph $G_s$ with 
$$V(G_s):=V(G)\cup \{v_1^s, \dots, v_r^s\}, \text{ and }\; E(G_s):=E(G) \cup \bigcup_{i=1}^r\{(v_{p_1(e_i)}, v_i^s), (v_i^s, v_{p_2(e_i)})\}. $$
We call the edges added by the big union the {\it supporting edges}. 
Let
$$O^*=O_{l_1} \dcup O_{l_2} \dcup \dots \dcup  O_{l_n} \; \text{ such that }\; n^2+n<l_1<l_2<\dots < l_n.$$
Let $D_s$ be a set of integers with 
\begin{equation}\label{60945}
|D_s|=r(=|E(G)|),\text{ and } x\in D_s \Rightarrow x> l_n.
\end{equation}
Let $s$ be a bijective map from $D_s$, satisfying \eqref{60945}, to $\{v_1^s, \dots, v_r^s\}$.
Let
$$O^*_s:=O^* \dcup  \dot{\bigcup_{x\in D_s}}O_x \; \text{ with } \;
V(O^*_s)=\{u_i^j: j\in\{l_1, \dots, l_n\}\cup D_s, \; 1\leq i \leq j\}.$$
Let $\alpha : C(O^*)\to V(G)$ be a bijective map. We define the digraph $(G\stackrel{\alpha}{\leftarrow} O^*)_s $ by  
$$(G\stackrel{\alpha}{\leftarrow} O^*)_s :=G_s\stackrel{\beta}{\leftarrow} O^*_s,\text{ where } \beta|_{C(O^*)}:= \alpha, \; \beta|_{\{O_x: x \in D_s\}}:=\{(O_x, s(x)): x \in D_s\}, $$
and say it is an {\it edge-supporting digraph for $G$}.
\end{definition}

\begin{remark}\label{65332}
Note that the definition of the edge-supporting digraphs includes a condition for the size of the circles of $O^*$. 
That condition is very important here, and was not present in \eqref{43959}.
We need to be cautious about this later on.
\end{remark}

\begin{nlemma}\label{} The following relation is definable.
\begin{equation}\label{756238429}
\{(O^*, G, G\stackrel{\alpha}{\leftarrow} O^*, (G\stackrel{\alpha}{\leftarrow} O^*)_s):
\text{$(G\stackrel{\alpha}{\leftarrow} O^*)_s$ is an edge-supporting digraph for $G$} \}
\end{equation}
\end{nlemma}

\begin{proof} 
The relation in question consists of those quadruples $(X_1,X_2,X_3, X_4)$ for which the highlighted conditions hold.
In some cases, there are explanations inserted between the conditions.
\begin{itemize}
\item There exists a quadruple $(X_1,X_2,Y,X_3)\in \eqref{43959}$, meaning $(X_1,X_2, X_3)$ is of the form $(O^*,G,G\stackrel{\alpha}{\leftarrow} O^*)$.
\end{itemize} 
Thus, instead of $(X_1,X_2, X_3)$, we use $(O^*,G,G\stackrel{\alpha}{\leftarrow} O^*)$ from now on in the proof.
Let $G$ have $n$ vertices. Now we are ready to shape $O^*$.
\begin{itemize}
\item $O_i \sqsubseteq O^*$ implies $i>n^2+n$.
\end{itemize} 
We turn to defining $X_4$ of the quadruple we started with.
\begin{itemize}
\item There exists a quadruple $(W_1, W_2, W_3, X_4) \in \eqref{43959}$, meaning $(W_1,X_4)$ is of the form $(O^*_s,G_s\stackrel{\beta}{\leftarrow} O^*_s)$. 
\end{itemize} 
At this point, $O^*_s$, $G_s$, and $\beta$ are just notations yet, we need additional conditions to make them be like in Definition \ref{71859}.
\begin{itemize}
\item $O^* \sqsubseteq O^*_s$
\item $O_i \sqsubseteq O^*_s$ implies $i\geq l_1$, where $l_1$ is the size of the smallest circle of $O^*$, as before.
\item $G\stackrel{\alpha}{\leftarrow} O^* \sqsubseteq X_4 (= G_s\stackrel{\beta}{\leftarrow} O^*_s)$.
\end{itemize}
The following conditions are to shape the supporting edges of our construction according to the definition.
\begin{itemize}
\item If $O_i\sqsubseteq O^*$ and $\male_i^L \sqsubseteq X_4$, then there exists $k>l_n$ for which $\male_i^L \leftrightarrow \male_k \sqsubseteq X_4$ holds.
Additionally, if $l$ is different from $i,k$, and $O_l \sqsubseteq O^*_s$, then there exists $\diamond \in \{\emptyset, L\}$ for which $\male_k \dcup \male_l^{\diamond} \sqsubseteq X_4$ holds.
\item If $O_i, O_j \sqsubseteq O^*$, $i\neq j$, and $\male_i^{\Box} \to \male_j^{\triangledown} \sqsubseteq X_4$ with $(\Box, \triangledown) \in \{\emptyset, L\}^2$, then there exists $k>l_n$ for which $\male_i^{\Box} \to \male_k \sqsubseteq X_4$ and $\male_k \to \male_j^{\triangledown} \sqsubseteq X_4$ both hold. 
Additionally, if $l$ is different from $i,j,k$, and $O_l \sqsubseteq O^*_s$, then there exists $\diamond \in \{\emptyset, L\}$ for which $\male_k \dcup \male_l^{\diamond} \sqsubseteq X_4$ holds.

\item If $O_i, O_j \sqsubseteq O^*$, $i\neq j$, and $\male_i^{\Box} \leftrightarrow \male_j^{\triangledown} \sqsubseteq X_4$ with some $(\Box, \triangledown) \in \{\emptyset, L\}^2$, then there exist two different $k_1, k_2>l_n$ for which all of
$$\male_i^{\Box} \to \male_{k_1}, \;
\male_{k_1} \to \male_j^{\triangledown}, \;
\male_j^{\triangledown}  \to \male_{k_2},\text{ and }
\male_{k_2} \to \male_i^{\Box}
$$
are substructures of $X_4$. 
Additionally, if $l$ is different from $i,j,k_i$, and $O_l \sqsubseteq O^*_s$, then there exists $\diamond \in \{\emptyset, L\}$ for which $\male_{k_i} \dcup \male_l^{\diamond} \sqsubseteq X_4$ holds for $i=1,2$.
\item If $O_k\sqsubseteq O_s^*$ and  $k>l_n$, then $k$ is one of the $k$s or $k_i$s of the previous three conditions.
\end{itemize} 
It is not hard to see that these conditions provide the structure we need.
 \end{proof}

We are finally ready to prove our main theorem. 

\begin{proof}[Proof of Theorem \ref{46031}] With \eqref{756238429}, fix a triple $(G, O^*,(G\stackrel{\alpha}{\leftarrow} O^*)_s)$, and let $n$ be the number of vertices of $G$.
We need to show that the set of digraphs embeddable into $G$ is definable. 
Let $X \sqsubseteq (G\stackrel{\alpha}{\leftarrow} O^*)_s$ and let $(G_X, O^*_X, G_X\stackrel{\gamma}{\leftarrow} O^*_X)$ be a triple consisting the second, first and third element of a 4-tuple of (\ref{756238429}) for which the following conditions hold. \\
Before listing the actual conditions being quite technical, it may be worth summarizing their goal plainly. 
What they do is link $G_X$ to both $X$ and $G$ the natural way, i. e. making sure that we have $G_X\leq G$ and we leave out the edges that got unsupported taking the substructure $X$. 

\begin{itemize}
\item $O_i \sqsubseteq O^*_X$ holds if and only if both $O_i\sqsubseteq O^*$, and $\male_i^{\Box} \sqsubseteq X$ for some $\Box \in \{\emptyset, L\}$ hold.

\item If $O_i, O_j \sqsubseteq O^*_X$, $i\neq j$, and $(\Box, \triangledown) \in \{\emptyset, L\}^2$, then

\begin{itemize}
\item $\male_i^{\Box} \dcup \male_j^{\triangledown} \sqsubseteq G_X\stackrel{\gamma}{\leftarrow} O^*_X$ holds if and only if one of the following three holds:
\begin{itemize}
\item $\male_i^{\Box} \dcup \male_j^{\triangledown} \sqsubseteq X$, or
\item $\male_i^{\Box} \to \male_j^{\triangledown} \sqsubseteq X$, but the edge is not supported in $X$, i. e.  there exists no $k>l_n$ (where $l_n$ is the size of the largest circle of $O^*$, as before) for which $\male_i^{\Box} \to \male_k \sqsubseteq X$ and $\male_k \to \male_j^{\triangledown} \sqsubseteq X$ both hold, or
\item $\male_i^{\Box} \leftrightarrow \male_j^{\triangledown} \sqsubseteq X$, but none of the two edges is supported in $X$.
\end{itemize}

\item $\male_i^{\Box} \to \male_j^{\triangledown} \sqsubseteq G_X\stackrel{\gamma}{\leftarrow} O^*_X$ holds if and only if one of the following two holds  
\begin{itemize}
\item $\male_i^{\Box} \to \male_j^{\triangledown} \sqsubseteq X$, and the edge is supported in $X$, or
\item $\male_i^{\Box} \leftrightarrow \male_j^{\triangledown} \sqsubseteq X$, but only the ``$i \to j$'' edge is supported in $X$.
\end{itemize}

\item $\male_i^{\Box} \leftrightarrow \male_j^{\triangledown} \sqsubseteq G_X\stackrel{\gamma}{\leftarrow} O^*_X$ holds if and only if $\male_i^{\Box} \leftrightarrow \male_j^{\triangledown} \sqsubseteq X$ and both edges are supported in $X$.
\end{itemize}
\end{itemize}
It is not hard to see that $G_X\leq G$ holds indeed, and all embeddable digraphs can be obtained this way.
 \end{proof}

\section{The remaining technicalities}

\begin{definition} 
The sum of the number of (both in- and out-)edges for a vertex, not counting the loops, is called the {\it loop-free degree} of the vertex. 
\end{definition}

\begin{nlemma}\label{45207} Let $0\leq p$ and $1\leq q$ be two fixed integers. We can define, with finitely many constants added to $(\mathcal{D}, \sqsubseteq)$, the set of digraphs that contain at most $p$ many vertices with loop-free degree at least $q$ each. 
\end{nlemma}

Before the easy proof, note that we can only use this lemma if we have a fixed constant, say $K=4$, for the whole paper, such that all usage of the lemma restricts to $p,q \leq K$. 
Otherwise there would be no guarantee we are using finitely many constants at all. Fortunately, $K=4$ will just do for the whole paper.

\begin{proof} Observe that the digraph $G$ has more than $p$ many vertices with at least $q$ loop-free degree each, if and only if it has an at most $(p+1)q$ element ``certificate'' substructure with the same property. Hence, by forbidding all those (finitely many) certificates, we define the set we need.
 \end{proof}

\begin{figure}[h]
\begin{center}
\begin{tikzpicture}[line cap=round,line join=round,>=triangle 45,x=0.5cm,y=0.5cm]
\clip(-0.5,-0.5) rectangle (20,2.5);
\large
\draw (5,1) node[anchor=south] {$\Box$};

\fill [color=black] (0,0) circle (1.5pt);
\fill [color=black] (0,2) circle (1.5pt);
\fill [color=black] (1,1) circle (1.5pt);
\fill [color=black] (3,1) circle (1.5pt);
\fill [color=black] (5,1) circle (1.5pt);

\draw [->]  (1,1) --  (3,1);
\draw [->]  (0,0) --  (1,1);
\draw [->]  (1,1) --  (0,2);
\draw [->]  (3,1) --  (5,1);

\draw (12,1) node[anchor=south] {$\Box$};
\draw (14,1) node[anchor=south] {$\triangledown$};

\fill [color=black] (7,0) circle (1.5pt);
\fill [color=black] (7,2) circle (1.5pt);
\fill [color=black] (8,1) circle (1.5pt);
\fill [color=black] (10,1) circle (1.5pt);
\fill [color=black] (12,1) circle (1.5pt);
\fill [color=black] (14,1) circle (1.5pt);
\fill [color=black] (16,1) circle (1.5pt);
\fill [color=black] (18,1) circle (1.5pt);
\fill [color=black] (19,0) circle (1.5pt);
\fill [color=black] (19,2) circle (1.5pt);

\draw [->]  (8,1) --  (10,1);
\draw [->]  (7,0) --  (8,1);
\draw [->]  (8,1) --  (7,2);
\draw [->]  (10,1) --  (12,1);
\draw [->]  (12,1) --  (14,1);
\draw [->]  (16,1) --  (14,1);
\draw [->]  (18,1) --  (16,1);
\draw [->]  (19,0) --  (18,1);
\draw [->]  (18,1) --  (19,2);
\end{tikzpicture}
\caption{}
\label{08496}
\end{center}
\end{figure}

\begin{proof}[Proof of Lemma \ref{25451}] Let us consider $E_i$ and $E_j$ given. We define the other components of the relation.

We start with $\male_i^{\Box}$ which is just the digraph $X$ for which
\begin{itemize}
\item $|V(X)|=i+2$.
\item $O_i \sqsubseteq X$.
\item We use Lemma \ref{45207} with $p=1$, and $q=3$, i. e. $X$ has at most one vertex with loop-free degree at least 3.
\item We use Lemma  \ref{45207} with $p=0$, and $q=4$ as well.
\item The first digraph of Fig. \ref{08496} is a substructure. The $\Box$ symbol is understood naturally, if $\Box=L$, then there is a loop there, if $\Box=\emptyset$, then there is not.
\item Depending on $\Box$,
\begin{itemize}
\item if $\Box=\emptyset$, then $O_i \dcup E_1 \sqsubseteq X$, that is the only cover of $O_i$ among the $IO$-graphs,
\item if $\Box=L$, then $O_i \dcup L_1 \sqsubseteq X$, that is definable with Lemma \ref{fkuerz8w9}.
\end{itemize}
\end{itemize}

We now start to deal with $\male_i^{\Box} \dcup \male_j^{\triangledown}$. $O_i\dcup O_j$ is the digraph with $i+j$ vertices that is a disjoint union of circles and both $O_i$ and $O_j$ are substructures. $\male_i^{\Box} \dcup \male_j^{\triangledown}$ is the digraph $X$ for which
\begin{itemize}
\item $|V(X)|=|V(\male_i^{\Box})|+|V(\male_j^{\triangledown})|$.
\item $\male_i^{\Box} \sqsubseteq X$, and $\male_j^{\triangledown} \sqsubseteq X$.
\item We use Lemma \ref{45207} with $p=2$, $q=3$ and with $p=0$, $q=4$.
\item Depending on $(\Box,\triangledown)$,
\begin{itemize}
\item if $(\Box,\triangledown)=(\emptyset, \emptyset)$, then $O_i\dcup O_j \dcup E_2 \sqsubseteq X$, which is just the digraph $Y$ for which
\begin{itemize}
\item $|V(Y)|=i+j+2$, and $O_i \dcup O_j \sqsubseteq X$,
\item $Y$ has the maximal substructure $E_k$ (among the ones with the previous property).
\end{itemize}
\item if $(\Box,\triangledown)=(L, \emptyset)$ or $(\emptyset, L)$, then $O_i\dcup O_j \dcup E_1 \dcup L_1 \sqsubseteq X$, which is just the digraph $Y$ for which
\begin{itemize}
\item $|V(Y)|=i+j+2$, and $O_i \dcup O_j \sqsubseteq X$,
\item $O_i\dcup O_j \dcup E_1$, which is the only $IO$-graph cover of $O_i\dcup O_j$, is a substructure,
\item $O_i\dcup O_j \dcup L_1$ is a substructure, and
\item on two elements, there is no substructure with both a loop and a loop-free edge.
\end{itemize}
\item if $(\Box,\triangledown)=(L, L)$ then $O_i\dcup O_j \dcup L_2 \sqsubseteq X$.
\end{itemize}
\end{itemize}

Now we turn to $\male_i^{\Box} \to \male_j^{\triangledown}$, which is just the digraph $X$ for which
\begin{itemize}
\item $|V(X)|=|V(\male_i^{\Box})|+|V(\male_j^{\triangledown})|$.
\item $\male_i^{\Box} \sqsubseteq X$, and $\male_j^{\triangledown} \sqsubseteq X$.
\item We use Lemma \ref{45207} with $p=2$, $q=3$ and with $p=0$, $q=4$.
\item The second digraph of Fig. \ref{08496} is substructure of $X$.
\end{itemize}

Finally, $\male_i^{\Box} \leftrightarrow \male_j^{\triangledown}$ is defined with the analogues of the  conditions for $\male_i^{\Box} \to \male_j^{\triangledown}$.
 \end{proof}

 \noindent {\it The end of the proof of Lemma \ref{26231}} \; To exclude the possible entanglement of the $\male$s, it is enough to forbid two types of digraphs as substructures. To introduce these two, first, we need two circles $O_i$ and $O_j$ with 
$$i\neq j,\;\;\;  V(O_i)=\{u_1, \dots, u_i\}, \text{ and } \; V(O_j)=\{v_1, \dots, v_j\}.$$  
Then we define $P_i$ and $P^{\triangledown}_{i,j}$ (where, as usual, $\triangledown\in \{\emptyset, L\}$) to be 
$$ V(P_i)=V(O_i) \cup \{u', u''\}, \;\;\; E(P_i)=E(O_i) \cup \{(u_1, u'), (u_1, u'')\},$$
and
$$V(P_{i,j})=V(O_i) \cup V(O_j) \cup \{u', v', w\}, $$
$$E(P_{i,j})=E(O_i) \cup E(O_j) \cup \{(u_1, u'), (u', w), (v_1, v'), (v', w)\},$$
with the usual, minor modification
$$V(P^{L}_{i,j})=V(P_{i,j}), \;\;\; E(P^{L}_{i,j})=E(P_{i,j})\cup \{(w,w)\}.$$

First, we define $P_i$ (using $O_i$). 
Actually, the very first, we define $P'_i$, that is just $P_i$ minus the vertex $u''$.
This is easy as $P'_i$ is the digraph one below $\male_i$ such that it has $O_i$ in it as a substructure but it does not equal $O_i \dcup E_1$. 
Now the sought $P_i=X$ can be defined by the following properties:
\begin{itemize}
\item $X$ covers $P'_i$,
\item $O_i \sqsubseteq X$,
\item $X$ has exactly the same three-element substructures as $P_i$, and
\item if $X$ covers $Y$ such that $O_i \sqsubseteq Y$, then $Y=P'_i$.
\end{itemize}  
 
Second, we define $P^{\triangledown}_{i,j}$ (using $O_i$ and $O_j$). It is the only digraph $X$ for which
\begin{itemize}
\item $|V(X)|=|V(O_i)|+|V(O_j)|+3$,
\item $O_i\dcup O_j \sqsubseteq X$,
\item $\male_i^\triangledown, \; \male_j^\triangledown \sqsubseteq X$, and
\item we use Lemma \ref{45207} with $p=2$, $q=3$.
\end{itemize}

Finally, to conclude the proof, we add two conditions to the ones already listed at the beginning of the proof:
\begin{itemize}
\item $O_i \sqsubseteq O^*$ implies $P_i \not\sqsubseteq X$,
\item $O_i, O_j \sqsubseteq O^*$ ($i\neq j$) implies $P^{\triangledown}_{i,j} \not\sqsubseteq X$ (for both $\triangledown\in \{\emptyset, L\}$).
\end{itemize}
\qed 

\medskip

 It is worth counting how big constants the usage of Lemma \ref{45207} requires. 
By its proof, it is clear that its usage with the pair $(p,q)$ requires constants of size at most $(p+1)q$. 
Looking back, we see that we used the lemma for the pairs $(0,4), (1, 4),$ and $(2,3)$. 
Hence the answer to our question is $9$, that is just $\max\{4,8,9\}$. \\

\medskip

{\bf Acknowledgements.}
The author is thankful to the anonymous reviewer who pointed out a serious flaw in the paper. 
In response, the author not just fixed the incorrect part, but also noticed an opportunity to free the paper from some inelegant and uncomfortable technicality.

\bibliographystyle{abbrv}      
\bibliography{references}

\bigskip

\noindent\rule{4cm}{0.4pt}

akunos@math.u-szeged.hu

\end{document}